\documentclass[11pt,oneside]{amsart}
\usepackage{graphicx}
\usepackage{psfrag}
\usepackage{amscd}
\usepackage{epsfig}
\usepackage{amsmath,ifthen, amsfonts, amssymb, amsopn,amsthm}

\usepackage{color}
\newtheorem{thm}{Theorem}[section]
\newtheorem{lem}[thm]{Lemma}

\newtheorem{cor}[thm]{Corollary}

\newtheorem{prop}[thm]{Proposition}

\newtheorem{fact}[thm]{Fact}

\theoremstyle{definition}
\newtheorem{defn}[thm]{Definition}

\newtheorem{prob}[thm]{Problem}

\begin{document}
\title[Hyperbolic Space and Its Medianization]{An Upper Bound for Hausdorff Distance Between Finite-Dimensional Hyperbolic Space and Its Medianization}
\author{Yongbin Zhou}

\begin{abstract}
We use de Sitter space to construct a concrete model for the measured wall structure of finite-dimensional hyperbolic space in hyperbolic model $\mathbb{I}^n$, which will induce a medianization $\mathcal M(\mathbb{I}^n)$. We get an upper bound for the Hausdorff distance between $\mathbb{I}^n$ and $\mathcal M(\mathbb{I}^n)$.

\end{abstract}

\maketitle

\section{Introduction}
We know that a measured wall structure induces a medianization. That is, given a measured wall structure on a space, we can construct a median space and isometrically embed the original space into the median space. We show that points in de Sitter space $\mathrm{dS}^n$ are in one-to-one correspondence with open half-spaces of $\mathbb{I}^n$, and define the medianization of $\mathbb{I}^n$ in a concrete way. As an application, we calculate an upper bound for the the Hausdorff distance between a finite-dimensional hyperbolic space and its medianization, which was proven to be finite in ~\cite{WALL}.

The \emph{Hausdorff distance} between two nonempty sets $A,B \subseteq (X,d)$ is
$
d_H(A,B)
=
\max\left\{
\sup_{a\in A}\inf_{b\in B} d(a,b),
\;
\sup_{b\in B}\inf_{a\in A} d(a,b)
\right\}
$
If $A \subseteq B$, then
$
d_H(A,B)
=
\sup_{b\in B}\inf_{a\in A} d(a,b).
$ So we describe how larger the medianization is than $\mathbb{I}^n$.

In this text, we often consider \emph{hyperboloid model} $\mathbb {I}^n$. We know that any totally geodesic embedded $\mathbb{I}^{n-1}$ is always an intersection between $\mathbb{I}^n$ and a $(n-1)$ dimension hyperplane in $\mathbb R^{n+1}$, and the embedded $\mathbb{I}^{n-1}$ (or hyperplane) will divide the $\mathbb{I}^n$ into two open half spaces. For every such open half-space $h$, the $h^c \in \mathbb I^{n}$ is a closed half-space. We consider half-spaces of $\mathbb{I}^n$ to be open and closed half-spaces obtained like this, and denote $\mathcal{H}$ the \emph{space of half-spaces}. $\{h,h^c\}$ is a partition of $\mathbb{I}^n$. We consider walls of $\mathbb{I}^n$ to be these partitions, and denote $\mathcal W$ the \emph{space of walls}. We say that the half-spaces $h$ and $h^c$ \emph{bound} the wall $\{h,h^c\}$, following the usage of ~\cite{WALL}. We also abuse the notation to say that the open half-space and closed half-space \emph{bound} this hyperplane. We say that two half-spaces $h_1$ and $h_2$ \emph{share} a wall if they bound the same embbed $\mathbb{I}^{n-1}$ (or hyperplane).

For any two points $x,y\in \mathbb{I}^n$, when an open half-space $h$ contains exactly one of them, we say $h$ saparates $x$ and $y$, and the wall $\{h,h^c\}$ saparates $x$ and $y$. Let $H(x|y)$ denote the set of half-spaces saparating $x$ and $y$, and $W(x|y)$ denote the set of walls saparating $x$ and $y$. We use $\triangle$ to note the symmetric differece operator for two sets.

Briefly speaking, we identify distance from an admissible section $\tau$ in medianization to a point section as an integration over the sphere $S^{n-1}$. The integrand is controlled by an odd function $b(\omega)$ with $\omega \in S^{n-1}$, and $\beta=\arcsin(\tanh b)$ is $1$-Lipschitz.
\begin{lem}[Section Distance and Border Function]
\label{lem:section distance and border function}
\[d(\tau,\sigma_x)=\mu_{\mathrm{dS}^n}(W^{border}_{\tau}\triangle \sigma_x)=\int_{S^{n-1}}
\int_0^{|b(\omega)|}
\cosh^{n-1}r
\,dr\,d\omega\] where the notation $\sigma_x$ is also abused to note the open half-spaces containing $x$, besides the point section.
\end{lem}
 Notice that one can let $b\ne0$ almost everywhere to make it large, just make sure it satisfies that 
\begin{thm}
\label{thm: border correspondence lip}
Elements of $\mathcal M(\mathbb{I}^n)$ are almost in one-to-one correspondence with the set of such $1$-Lipschitz odd functions.
\end{thm}
Another thing used is that the integration can be arbitrarily large if we move the compared point far enough, so we can take the point to minimize the integration, and this condition gives an equation
\begin{lem}[Locally Minimum Condition]
\label{lem:locally minimum condition}
When $n\ge$ 2, after we choose that point $x$ minimizing $d(\tau_b,\sigma_x)$, and $b\ne 0$ almost everywhere, we get
\[
\int_{S^{n-1}}\operatorname{sgn}(b(\omega))\,\omega\,d\omega=0\]

\end{lem}
This would turn the upper bound question into an analysis form:
\begin{prob}
\label{prob:anal form}
Let \(n\ge 2\), and $b:S^{n-1}\to\mathbb R$ be a finite-valued odd function. Assume that:
\begin{enumerate}
\item $\beta(\omega)=\arcsin(\tanh b(\omega))$ is \(1\)-Lipschitz on \(S^{n-1}\) with respect to the standard spherical metric.
\item ($b\ne 0\ a.e.$ and) \(b\) satisfies the locally minimum condition
\[
\int_{S^{n-1}}
\operatorname{sgn}(b(\omega))
\omega
d\omega
=
0.
\]
\end{enumerate}

Consider the functional
\[
\mathcal E(b)
=
\int_{S^{n-1}}
\left(
\int_{0}^{|b(\omega)|}
\cosh^{\,n-1}(r)\,dr
\right)
d\omega.
\]

Does there exist a constant \(C_n<\infty\), depending only on \(n\), such that
$
\mathcal E(b)\le C_n
$
for every function \(b\) satisfying the above assumptions?
\end{prob}

\begin{thm}
\label{thm:hausdorff bound}
\[
d_\mathcal M(\mathbb{I}^n,\mathcal M(\mathbb{I}^n))
\le
\sqrt{\pi} 
\frac{n-1}{2}
\, \frac{\Gamma\left(\frac{n-1}{2}\right)}{\Gamma\left(\frac{n}{2}\right)}
\int_0^{
\operatorname{artanh}(2^{-1/(n-1)})
}
\cosh^{n-1}r
\,dr
\]
\end{thm}
The right-hand side gives a choice of such $C_n$.
One may want to read the discussion section first to understand this passage easily and quickly. The concrete construction of the measured wall structure for finite-dimensional hyperbolic space is inspired by ~\cite{NR}.

\section{The de Sitter Construction}
\subsection{Preliminary}
We give an informal instruction to measured wall geometry and medianization.
\subsubsection{Measured Wall Geometry}
A \emph{space with measured walls} consists of a set $X$, a collection $\mathcal W$ of walls, a $\sigma$-algebra $\mathcal B$ of subsets of $\mathcal W$, and a measure $\mu$ on $\mathcal B$, such that for every pair of points $x,y\in X$, the set of walls separating $x$ and $y$, denoted by $\mathcal W(x|y)$, belongs to $\mathcal B$ and has finite measure. A wall is a partition ${h_1,h_2}$ of the space $X$, and $h_1$ and $h_2$ are called half-spaces. If $X$ is a geodesic metric space, we can have up to a consistant constant factor $
d(x,y)=\mu(\mathcal W(x|y)).
$

We call
$
(X,\mathcal W,\mathcal B,\mu)
$
a measured wall structure on $X$.

\subsubsection{Medianization Geometry}

Let $(X,d)$ be a (pseudo-)metric space. Given three points
\(x_1,x_2,x_3\in X\), a point $m\in X$ is called a \emph{median point}
if
$
d(x_i,x_j)
=
d(x_i,m)+d(m,x_j)
$
for every pair $i\neq j$.

If for every triple of points the set of median points has diameter
zero, then $X$ is called a \emph{median (pseudo-)metric space}. A natural
question is whether a given space can be embedded isometrically into a
median space. If we have a measured wall structure, we can construct such a median space.

Suppose that
$
(X,\mathcal W,\mathcal B,\mu)
$ is a measured wall structure. Let $\mathcal H$ denote the collection of half-spaces associated with the walls.

A section is a map choosing a half-space for every wall, or equivalently a set of half-spaces noting which it chooses, or equivalently a function on the set of half-spaces $\sigma:\mathcal{H}\to\{1,-1\}$, or on the set the of walls when the corresponding between half-spaces and $\{1,-1\}$ is defined. All definitions are up to a set where the section abstains, and usually we can expect the set to be of measure-0.

An admissible section is a section that, if it chooses a half-space $h$, and there is another half-space $h'\subset h$, then it chooses $h'$.

For a given point, we can construct a corresponding admissible section, which in this paper is called point section. A point section of a point $x$ is a section $\sigma_x$ induced in this way: if one of a pair of half-space contains $x$, then $\sigma_x$ chooses the one containing $x$.

\begin{prop}
A point section is an admissible section.
\end{prop}

We can compute the distance between two sections by the measure of the set of walls where they disagree. We consider the set of admissible sections in $\mathcal{B}$ that is at a finite distance from a point section, these sections form the \textbf{medianization} $\mathcal{M}(X)$.

By viewing distance between two sections as an integration on the wall space, we can see that $\mathcal{M}(X)$ is a subspace of $L^1(\mathcal W,\mu)$. It is a median subspace of $L^1(\mathcal W,\mu)$. That is, for any three points in $\mathcal{M}(X)$, we consider the set of median points, and there is always a median point just in $\mathcal{M}(X)$.The median section of three sections is by taking the majority choice on a wall.

\subsubsection{Construction of the medianization of a space with measured walls.}
Let
$
(X,\mathcal W,\mathcal B,\mu)
$
be a space with measured walls. We show how to construct a associated median space
$
\mathcal M(X)
$
and an isometric embedding
$
\iota:X\hookrightarrow \mathcal M(X).
$

We have a pulled back measure $\mu_\mathcal H$ and $\sigma$-algebra $
\mathcal B^{\mathcal H}
$ on the space of half-spaces $\mathcal H$.
Fix a base point
$
x_0\in X.
$
Consider
\[
\mathcal B^{\mathcal H}_{\sigma_{x_0}}
=
\Bigl\{
A\subset\mathcal H:
A\triangle \sigma_{x_0}
\in
\mathcal B^{\mathcal H},
\;
\mu_\mathcal H(A\triangle \sigma_{x_0})<\infty
\Bigr\}
\]

This set has a pseudometric
$
d(A,B)
=
\mu_\mathcal H(A\triangle B).
$
By Example 2.8 of ~\cite{KAmedian}, this is a median pseudometric space, and the map
$
A\mapsto \chi_{A\triangle\sigma_{x_0}}
$
embeds it isometrically into
$
L^1(\mathcal H,\mu).
$
The median operation is given by counting the majority:
$
m(A,B,C)
=
(A\cap B)
\cup
(A\cap C)
\cup
(B\cap C)
$.
Moreover, it is proved that this set is identitied when changing the base point. So we can just write 
$\mathcal B_X^{\mathcal H}$.

Define
$
\mathcal M(X)
=
\{
A\in
\mathcal B_X^{\mathcal H}
:
A
\text{ is an admissible section}
\}.
$
Proposition 3.14 of ~\cite{KAmedian} shows that
$
\mathcal M(X)
$
is a median subspace of
$
\mathcal B_X^{\mathcal H}.
$ The explanation of median subspace will be given later.

Define
$
\iota:X\to\mathcal M(X),
\ 
x\mapsto\sigma_x.
$
To prove that this is an isometric embedding, observe that for
$
x,y\in X,
$
the symmetric difference
$
\sigma_x\triangle\sigma_y
$
contains exactly the half-spaces on which the two sections disagree.
Applying the projection $p$ gives
$
p(\sigma_x\triangle\sigma_y)
=
\mathcal W(x|y),
$
The half-space version is noted as $
\mathcal H(x|y)
$
So
$
\mu(\sigma_x\triangle\sigma_y)
=
\mu(\mathcal W(x|y)).
$

But the wall metric on $X$ is defined by
$
d_X(x,y)
=
\mu(\mathcal W(x|y)).
$
Therefore
$
d_{\mathcal M(X)}
(\sigma_x,\sigma_y)
=
d_X(x,y),
$
and so
$
\iota
$
is an isometric embedding.

Consequently,
$
(X,d_X)
\hookrightarrow
(\mathcal M(X),d_{\mathcal M(X)})
$
isometrically.
The space $\mathcal M(X)$ is called the median space associated with
$(X,\mathcal W,\mathcal B,\mu)$.


Note that ${\mathcal M(X)}$ is not usually a median space. It is a median subspace of $\mathcal B^{\mathcal H}_X$, which means given 3 distinct points, the set of their median points has measure zero and has non-empty intersection with $\mathcal M(X)$, but maybe non-empty intersection with $\mathcal B^{\mathcal H}_X$. However, for real hyperbolic spaces, their medianizations are median spaces.

\subsection{Hyperplanes and $\mathrm{dS}^n$}
In this part, we introduce the $\mathrm{dS}^n$ space to modelize the hyperplanes.

On \(\mathbb R^{n+1}\) let
$
\langle x,y\rangle = -x_0y_0+x_1y_1+\cdots+x_ny_n .
$
The hyperboloid model of hyperbolic \(n\)-space is (we use $\mathbb I$ to distinguish from half-spaces)
\[
\mathbb{I}^n
=
\{x\in \mathbb R^{n+1}:\langle x,x\rangle=-1,\ x_0>0\}.
\]

Let
$
\mathrm{dS}^n
=
\{u\in \mathbb R^{n+1}:\langle u,u\rangle=1\}
$
be de Sitter space. For each \(u\in \mathrm{dS}^n\), define
$
P_u
=
\{x\in \mathbb{I}^n:\langle x,u\rangle=0\}.
$
The orthogonal complement \(u^\perp\) has quadratic form signature \((n-1,1)\), so
$
P_u=u^\perp\cap \mathbb{I}^n
$
is a totally geodesic subspace of \(\mathbb{I}^n\) of dimension \(n-1\).

On the other hand, every totally geodesic hyperplane \(P\subset \mathbb{I}^n\) is of the form
$
P=\mathbb{I}^n\cap V
$
for a \(n\)-dimensional linear subspace \(V\subset \mathbb R^{n+1}\). Its quadratic form orthogonal complement \(V^\perp\) is one-dimensional, so there exists \(u\in\mathbb R^{n+1}\), unique up to a scalar, such that
$
V^\perp=\mathbb R u,$ and $\langle u,u\rangle=1.
$
Thus
$
P=P_u.
$

Therefore $\mathrm{dS}^n$ is the space of open half-spaces. We define that an open half-space $u\in \mathrm{dS}^n$ contains a point $x\in \mathbb{I}^n$ if and only if $\ \langle x,u \rangle \ >\ 0.$ One may choose the other one, but we choose the positive one here. Moreover, the space of totally geodesic ($n-1$)-dimensional hyperbolic hyperplanes in \(\mathbb{I}^n\) is identified with \(\mathrm{dS}^n\) up to sign, that is
$
\Omega
\cong
\mathrm{dS}^n/\{\pm 1\}.
$
Then \(u\) and \(-u\) determine the same hyperplane. In fact, $u$ and $-u$ are the two open half-spaces obtained by dividing $\mathbb{I}^n$ with a same hyperplane. 

We can parametrize de Sitter space by
\[
u=(\sinh r,(\cosh r)\,\omega),
\qquad
r\in \mathbb R,\quad \omega\in S^{n-1}.
\]
Verification:
$
\langle u,u\rangle
=
-(\sinh r)^2+(\cosh r)^2|\omega|^2
=
-(\sinh r)^2+(\cosh r)^2
=
1,
$
where $|\omega|=1$, and $\sinh$ is a bijection from $\mathbb{R}$ to itself.
So
$
\mathrm{dS}^n\cong \mathbb R\times S^{n-1}.
$

\subsection{G-Invariant Measure}
In this part, we show that the measure is G-invariant.

The parametrization
$
u(r,\omega)
=
(\sinh r,\cosh r\,\omega),
\quad
r\in\mathbb R,\quad \omega\in S^{n-1},
$
results in
$
\langle \partial_r u,\partial_r u\rangle=-1.
$
For tangent vectors $v_1,v_2\in T_\omega S^{n-1}$,
\[
\langle d_\omega u(r,v_1),d_\omega u(r,v_2)\rangle
=
\cosh^2 r\, g_{S^{n-1}}(v_1,v_2)
\]
where $g_{S^{n-1}}$ is the standard Riemannian metric tensor on $S^{n-1}$. We get the induced tensor
$
g_{\mathrm{dS}}
=
-dr^2+\cosh^2 r\,g_{S^{n-1}}.
$
This tensor is not a Riemannian metric tensor. However, what we care about essentially is the measure.

Take
$
|d\operatorname{vol}_{\mathrm{dS}}|
=
\sqrt{|\det g_{\mathrm{dS}}|}\,dr\,d\omega.
$ By
$
g_{\mathrm{dS}} =
\begin{pmatrix}
-1 & 0_{1\times(n-1)} \\
0_{(n-1)\times 1} & \cosh^2 r\, g_{S^{n-1}}
\end{pmatrix}
$, 
we get
$
\det g_{\mathrm{dS}} = -\cosh^{2(n-1)}r\,\det g_{S^{n-1}},
$
so
$
|d\operatorname{vol}_{\mathrm{dS}}| = \cosh^{n-1}r\,dr\,d\omega.
$

Therefore
$
d\mu(r,\omega)
=
\cosh^{n-1}r\,dr\,d\omega
$. It is $SO(n,1)$-invariant on de Sitter space, equivalently on the space of totally geodesic hyperplanes in $\mathbb{I}^n$.


\subsection{Equal to Distance}
We show that the measure of distance is equal to the distance of two points.

The measure is G-invariant, and $G$ acts transitively, so the function
$
F(x,y) = \mu(W(x|y))
$
only relies on $d_\mathcal{M}(x,y)$. A wall divides them if and only if the open half-space contain exactly one of them, equivalent for the closed one.

There exists a function
$
f : [0,\infty) \to [0,\infty)
$
such that
$
F(x,y) = f(d_\mathcal M(x,y)).
$

Suppose that $z$ lies on the geodesic between $x$ and $y$. A hyperplane(or half-space) separates $x$ from $y$ if and only if it separates $x$ from $z$ or separates $z$ from $y$, and these two cases cannot hold at the same time, up to a set of measure zero. So $
F(x,y) = F(x,z) + F(z,y)$.
So $
f(a+b) = f(a) + f(b)$.

It is obvious that $f$ is measurable and finite on any bounded interval, so it must be linear. So
$
f(t) = ct
$
for some constant $c=c(n)\ge0$. And $c>0$ is obvious: by taking $|r|$ small and $\omega$ ranging in a small area on $S^{n-1}$, we can separate $o=(1,\dots,0)$ and another point very far away. One can see ~\cite{SWT}.

\subsection{Computation of the Constant}
\label{subsection:Computation of the Constant}
We compute the positive scalar relating the measure of separating
hyperplanes to the hyperbolic distance.

Let
$
o=(1,0,\ldots,0),
\ 
x_t=(\cosh t,\sinh t,0,\ldots,0), t>0.
$
Then
$
-\langle o,x_t\rangle=\cosh t,
$
so
$
d_{\mathbb \mathbb{I}^n}(o,x_t)=t.
$

Let
$
u=(\sinh r,\cosh r\,\omega),
\,
r\in \mathbb R,\, \omega\in S^{n-1}.
$
The corresponding open half-space is
$
H_u=\{x\in \mathbb \mathbb{I}^n:\langle u,x\rangle>0\}.
$

$
o\in h_u
\iff
r<0,\qquad
x_t\in h_u
\iff
\tanh r<\omega_1\tanh t.
$

Therefore \(u\) separates \(o\) and \(x_t\) exactly when
\[
\operatorname{sign}(r)=\operatorname{sign}(\omega_1),
\qquad
|r|<
\operatorname{arctanh}(|\omega_1|\tanh t).
\]

So the separating set is
\[
A_t
=
\left\{
(r,\omega)\in \mathbb R\times S^{n-1}
:
\operatorname{sign}(r)=\operatorname{sign}(\omega_1),
\;
|r|<
\operatorname{arctanh}(|\omega_1|\tanh t)
\right\}.
\]

The invariant measure on \(\mathrm{dS}^n\) is
$
d\mu_{dS}
=
\cosh^{n-1}r\,dr\,d\omega .
$

So
$
\mu_{\mathrm{dS}^n}(A_t)
=
2
\int_{\omega_1>0}
\int_0^{\operatorname{arctanh}(\omega_1\tanh t)}
\cosh^{n-1}r\,dr\,d\omega .
$

Differentiating with respect to \(t\),
$
\frac{d}{dt}\mu_{\mathrm{dS}^n}(A_t)
=
\frac{2\,\operatorname{vol}(S^{n-2})}{n-1}.
$
Since
$
\mu_{\mathrm{dS}^n}(A_0)=0
$, 
we get
$
\mu_{\mathrm{dS}^n}(A_t)
=
\frac{2\,\operatorname{vol}(S^{n-2})}{n-1}\,t.
$

Because \(d_{\mathbb \mathbb{I}^n}(o,x_t)=t\), we get
$
\mu_{\mathrm{dS}^n}(A_t)
=
\frac{2\,\operatorname{vol}(S^{n-2})}{n-1}
\,d_{\mathbb \mathbb{I}^n}(o,x_t)
.
$

\begin{thm}
For points $x,\ y \in \mathbb{I}^n$,
\[
\mu_{\mathrm{dS}^n}(H(x|y)\})
=
\frac{2\,\operatorname{vol}(S^{n-2})}{n-1}
\,d_{\mathbb \mathbb{I}^n}(o,x_t)
\]
\end{thm}

To make the wall measure equal to hyperbolic distance, we set
\begin{defn}
\[
\mu_{\mathcal W}
=
\frac{n-1}{2\,\operatorname{vol}(S^{n-2})}
\,\mu_{\mathrm{dS}^n}
\]
\end{defn}

\section{Wall Geometry}

\subsection{Point-containing Condition for Walls}
Given a hyperbolic point $x\in \mathbb{I}^n$ and a half-space $u\in \mathrm{dS}^n$, we discuss how to see if $x$ is on the wall represented by $u$.
Recall that the hyperplane $P_u$ determines two half-spaces:
\[
h_u^+
=
\{x\in \mathbb{I}^n:
\langle x,u\rangle>0
\}
\text{ and }
h_u^-
=
\{x\in \mathbb{I}^n:
\langle x,u\rangle<0
\}.
\]
Observe that
$
h_u^-
=
h_{-u}^+.
$ So the space of totally geodesic codimension-1 hyperbolic hyperplanes is
$
\Omega
=
\mathrm{dS}^n/\{\pm1\}.
$

\begin{prop}
For a point $x\in \mathbb{I}^n$, the set of open half-spaces bounding a hyperplane passing through $x$, that is the set of open half-spaces obtained by a hyperplane containing $x$, is
\[
\Omega_x
=
\{[u]\in\mathrm{dS}^n:
\langle x,u\rangle=0
\}
\]
With
$
\Omega_x
=
x^\perp\cap\mathrm{dS}^n\cong S^{n-1}.
$

\end{prop}

\begin{proof}
Since
$
\langle x,x\rangle=-1,
$
the orthogonal complement $x^\perp$ has positive definite quadratic form and dimension $n$. Therefore
$
x^\perp\cap\mathrm{dS}^n
=
\{u\in x^\perp:
\langle u,u\rangle=1\}
$
is the unit sphere in $x^\perp$.
\end{proof}
In fact, by a $SO(n,1)$ action, one only needs to check a base point $o=(1,0,...,0)$. We identify $h_u^+$ with u. One can choose $h_u^-$, however we use $h_u^+$ here.

\subsection{Intersection Condition for Walls}
Given two open half-spaces $u,v\in \mathrm{dS}^n$, we discuss how to see if the walls represented by $u$ and $v$ intersect.
They determine hyperbolic walls
$
P_u
=
\mathbb{I}^n\cap u^\perp,
\ 
P_v
=
\mathbb{I}^n\cap v^\perp.
$

Then
$
P_u\cap P_v
=
\mathbb{I}^n\cap u^\perp\cap v^\perp
=
\mathbb{I}^n\cap \operatorname{span}(u,v)^\perp.
$
Thus the question reduces to determining when
$
\operatorname{span}(u,v)^\perp
$
contains a point of \(\mathbb{I}^n\).

Observe that \(x\in \mathbb{I}^n\) if and only if \(x\) is a future-pointing timelike\footnote{Physicists call the first coordinate time.} vector normalized by
$
\langle x,x\rangle=-1.
$
So
$
P_u\cap P_v\neq\varnothing
$
if and only if \(\operatorname{span}(u,v)^\perp\) contains a nonzero timelike vector \(x\), namely
$
\langle x,x\rangle<0.
$

Let
$
a=\langle u,v\rangle.
$
Since
$
\langle u,u\rangle
=
\langle v,v\rangle
=
1,
$
the Gram matrix\footnote{See discussion part.} of \(\operatorname{span}(u,v)\) is
$
G(u,v)
=
\begin{pmatrix}
1 & a\\
a & 1
\end{pmatrix}.
$
Its determinant is
$
\det G
=
1-a^2.
$

The orthogonal complement contains a timelike vector precisely when
\(\operatorname{span}(u,v)\) is spacelike\footnote{Physicists call the rest coordinates space.}, namely when \(G(u,v)\) is positive definite. Since the diagonal entries are positive, this is equivalent to
$
1-a^2>0.
$
Using the parametrization
$
u
=
(\sinh r,\cosh r\,\omega),\ 
v
=
(\sinh s,\cosh s\,\eta),\ 
\omega,\ \eta\in S^{n-1},
$
we compute
$
\langle u,v\rangle
=
-\sinh r\,\sinh s
+
\cosh r\,\cosh s\,(\omega\cdot\eta).
$

So the intersection condition becomes
\[
{
\left|
-\sinh r\,\sinh s
+
\cosh r\,\cosh s\,(\omega\cdot\eta)
\right|
<
1.
}
\]

Moreover, we get other cases.
$
|\langle u,v\rangle|
=
1
$
corresponds to walls that meet only at infinity, except for $u=v$ and $u=-v$.
$
|\langle u,v\rangle|
>
1
$
corresponds to unparallel walls.

Since \(u\) and \(-u\) share a hyperplane, the condition is invariant under changing signs, which explains the appearance of the absolute value.

Therefore
$
P_u\cap P_v\neq\varnothing
\iff
|\langle u,v\rangle|<1\,\text{ or }\,u=\pm v.
$

A point and its opposite have measure zero. So for fixed $u\in\mathrm{dS}^n $, we have 
$\begin{aligned}
&\mu_{\mathrm{dS}^n}(\{v\in\mathrm{dS}^n:P_u\cap P_v\neq\varnothing\})
\\=
&\mu_{\mathrm{dS}^n}(\{v\in\mathrm{dS}^n:
\left|
-\sinh r\,\sinh s
+
\cosh r\,\cosh s\,(\omega\cdot\eta)
\right|
<
1\})
\end{aligned}$


\subsection{Inclusion Condition for Half-spaces}
We discuss when an open half-space contains another one.

Recall that for every point
$
u\in\mathrm{dS}^n,
$
define the half-space corresponding $u$ to be
$
H_u
=h_u^+=
\{x\in \mathbb{I}^n:
\langle x,u\rangle>0
\}.
$

We have $
H_u\subseteq H_v
\iff
u-v\in\overline{\mathcal C}^{+},
$
where\footnote{Physics called it future light cone.}
\[
\overline{\mathcal C}^{+}
=
\{z\in\mathbb R^{n+1}:
\langle z,z\rangle\le0,
\ z_0\ge0
\}.
\]
\begin{prop}
\label{prop:Half-Space Inclusion}
$
H_u\subseteq H_v
\iff
u_0-v_0\ge0
\quad\text{and}\quad
\langle u,v\rangle\ge1.
$
\end{prop}

\begin{proof}
When $\langle u,v\rangle\ \in\ (-1,1)$, by intersection condition for walls, the inclusion of half-spaces is impossible. So we only need to consider $|\langle u,v\rangle|\ge1$.

\smallskip
Assuming that
$
\langle u,v\rangle\ge 1
\text{ and }
u_0-v_0\ge 0,
$
we prove that
$
H_u\subseteq H_v.
$

Since
$
\langle u,u\rangle=\langle v,v\rangle=1,
$
we have
$
\langle u-v,u-v\rangle\le 0 \text{ and } 
(u-v)_0\ge 0.
$ Now let \(x\in \mathbb{I}^n\). Write
$
x=(x_0,x') \text{ and } u-v=(a,A),
$
where \(x'\in\mathbb R^n\) and \(A\in\mathbb R^n\).

By
$
x_0^2=1+|x'|^2,
$
we get
$
x_0>|x'|.
$
By
$
-a^2+|A|^2\le 0,
$
we get
$
a^2\ge |A|^2.
$
Because \(a\ge 0\), we get
$
a\ge |A|.
$
Therefore
$
\langle x,u-v\rangle
=
-x_0a+x'\cdot A
\le
-x_0a+|x'||A|
\le
-x_0a+|x'|a
=
-(x_0-|x'|)a
\le 0.
$ Now suppose \(x\in H_u\). Then
$
\langle x,u\rangle>0.
$
Thus
$
\langle x,v\rangle
=
\langle x,u\rangle-\langle x,u-v\rangle
>
0+0
=
0.
$
So \(x\in H_v\), that is $
H_u\subseteq H_v.
$

\smallskip
For the other direction, when $n=1$ it is obvious. Assume that $n\ge2$.

\smallskip
First for $\langle u,v \rangle \ge 1$, by an $SO(n,1)$ transformation we only need to consider $u=(0,1,0,\cdots,0)$. The condition becomes when $x_1>0$ we have $\langle x,v\rangle >0$. Note $v'=(v_2,\cdots,v_n\in \mathbb R^{n-1})$.

Take $x=(\sqrt{1+|\omega|^2+s^2},s,\omega),\ \omega\in\mathbb R^{n-1},\ s>0$. We have $\langle x,v \rangle = -v_0\sqrt{1+|\omega|^2+s^2} +v_1s+\omega\cdot v'>0$. Taking $s\to0^+$, we have $\langle x,v \rangle \to -v_0\sqrt{1+|\omega|^2}+\omega\cdot v'\ge0$. So $-v_0\ge|v'|$. So $|v_1|\ge1$. We show that $v_1\le-1$ is impossible. We have shown that $-v_0\ge|v'|$. If $|v'|=0$, let $|\omega|=0$ then we have $\langle x,v \rangle =\frac{-s^2+v_0^2}{-v_0\sqrt{1+s^2}-v_1s}>0$, which is impossible for large $s$.

If $|v'|\ne0$, let $\omega=-\frac{r}{|v'|}v'$ and $f_s(r)= -v_0\sqrt{1+|\omega|^2+s^2}+v_1s -r|v'|>0$ for $r\in\mathbb R$. If $-v_0 = |v'|$, let $r\to+\infty$, then we get $f_s(r)\to v_1s<0$, impossible. If $-v_0 > |v'|$, $f_s^{'}(r)=\frac{-v_0r}{\sqrt{1+r^2+s^2}}-|v'|$. $f_s$ reaches its minimum on $\mathbb R$ when $r=\sqrt\frac{1+s^2}{v_0^2-|v'|^2}\ |v'|$, and we have this minimum $\sqrt{(v_1^2-1)(1+s^2)}\ +\ v_1s\ \ge0$, which is equivalent to $-s^2-1+v_1^2\ge0$, and it is impossible for large $s$.

So $v_1 \le-1 $ is impossible. We have $\langle u,v \rangle=v_1\ge1$, which is invariant under $SO(1)$ transformation. 

\smallskip
Now for $ u_0 \ge v_0 $, considering an $SO(n)$ action rotating the last $n$ coordinates, we can assume that $u=(\sinh a, \cosh a,0,\cdots,0),\ a\in\mathbb R$, and $v=(\bar{v_0},\bar{v_1},\bar{v'}),\ \bar{v'}\in\mathbb R^{n-1}.$ Under a $SO(n,1)$ action of $\begin{pmatrix}
A_a & 0 \\
0 & I_{n-1}
\end{pmatrix}$ $u$ becomes $u=(0,1,0,\cdots,0)$ and $v$ becomes $({v_0},{v_1},{v'})$ as above, where $A_a = \begin{pmatrix}
\cosh a & -\sinh a \\
-\sinh a & \cosh a
\end{pmatrix}$. We have $\bar{v_0}=v_0\cosh a+{v_1}\sinh a$. We have shown that $v_1\ge1$ and $-v_0\ge|v'|$. Plus that $v_1\le\sqrt{1+v_0^2}\le1-v_0$. If $a\ge0$, we have 
\[
\begin{aligned}\bar{v_0}&=v_0\cosh a+v_1 \sinh a\\&\le v_0\cosh a +(1-v_0)\sinh a \\&= \sinh a + v_0(\cosh a - \sinh a)\le \sinh a\end{aligned}
\]
If $a<0$, we have $v_1\sinh a\le \sinh a$ and $v_0\cosh a\le 0 $, so
\[
\bar{v_0}=v_0\cosh a+v_1\sinh a\le\sinh a
\]
\end{proof}

\subsection{Modified Inclusion Condition for Half-spaces} $\mathrm{dS}^n$ has another kind of parametrization, which we will use later.

We have parametrized de Sitter space by
$
u(r,\omega)
=
(\sinh r,\cosh r\,\omega),
\,
r\in\mathbb R,\,\omega\in S^{n-1}.
$
Now introduce the coordinate
$
\rho
=
\arcsin(\tanh r)\in (-\frac\pi2,\frac\pi2).
$
Then
$
\tanh r=\sin\rho,
\,
\operatorname{sech}r=\cos\rho,
$
and
$
\sinh r=\tan\rho,
\,
\cosh r=\sec\rho.
$
Therefore
$
u(\rho,\omega)
=
(\tan\rho,\sec\rho\,\omega).
$

Take
$
u=u(\rho,\omega),
\,
v=u(\rho',\eta),
$
and let
$
\alpha=d_{S^{n-1}}(\omega,\eta),
\qquad
\omega\cdot\eta=\cos\alpha.
$
The condition
$
v_0-u_0\ge 0
$
becomes
$
\tan\rho'-\tan\rho\ge 0,
$
which is equivalent to
$
\rho'\ge \rho.
$

On the other hand,
\[
\begin{aligned}
\langle u,v\rangle
&=
-\tan\rho\,\tan\rho'
+
\sec\rho\,\sec\rho'\,(\omega\cdot\eta) \\
&=
\frac{\cos\alpha-\sin\rho\sin\rho'}
{\cos\rho\cos\rho'}.
\end{aligned}
\]

Thus
$
\langle u,v\rangle\ge 1
$
is equivalent to
$
\cos\alpha \ge  \cos(\rho'-\rho)
$
Since
$
0\le \alpha\le \pi,
$
and
$
\pi\ge\rho'-\rho\ge 0,
$
so this is equivalent to
$
\alpha\le \rho'-\rho.
$

\begin{prop}
$
H_{u(\rho,\omega)}
\subseteq
H_{u(\rho',\eta)}
\iff
\rho'-\rho
\ge
d_{S^{n-1}}(\omega,\eta)
$
\end{prop}

\section{Medianization Geometry}

\subsection{Sections}
Recall the equivalant definitions of section. Especially, we can view a section as a set of open half-spaces $\sigma\subset \mathrm{dS}^n$, such that $\sigma\{h,h^c\}=h$ for an open half-space $h$ if and only if $h\in \sigma$. We can use a number to note this. Also, a section can be represented by a sign function
$
\varepsilon:
\mathrm{dS}^n
\longrightarrow
\{\pm1\}
$, such that $h\in\sigma \iff \varepsilon(h)=1$.

A trival example is taking all the open half-spaces. Another example is point section: for a point $x\in \mathbb{I}^n$, define
$
\sigma_x
=
\{u\in\mathrm{dS}^n:
\langle x,u\rangle>0
\}.
$
In this case, $
\varepsilon_x(u)
=
\operatorname{sgn}
(\langle x,u\rangle),\ a.e.
$
The set
$
\{u:
\langle x,u\rangle=0\}
$
has measure zero, and so does not affect the measured wall structure. In fact, if we consider the equivalent classes of sections up to choice on a zero-measure set of walls, every point section lies in a distinct class.

Recall from section \ref{subsection:Computation of the Constant} that, a point section is defined as follow and will be used in later section. For
$
x_t=(\cosh t,\sinh t,0,\ldots,0), t>0.
$
\[
\sigma_{x_t}={\left\{(\sinh r,\cosh r \omega)\in \mathrm{dS}^n:\omega_1>\frac{\tanh r}{\tanh t}\right\}}
\]

\subsection{Medianization of Real Hyperbolic Space Is a Median Space}

\begin{thm}
\(\mathcal M(\mathbb{I}^n)\) is a median pseudo-metric space.
\end{thm}

\begin{proof}
We use the concrete de Sitter parametrization of the walls of \(\mathbb{I}^n\).

Recall that
$
h_u^+
=
\{x\in \mathbb{I}^n:\langle x,u\rangle> 0\},
\ 
h_u^-
=
\{x\in \mathbb{I}^n:\langle x,u\rangle< 0\}.
$

A section may therefore be represented by a sign function
$
\varepsilon:\mathrm{dS}^n\to\{\pm1\}
$.
The meaning is that
\[
\varepsilon(u)=+1
\quad\Longleftrightarrow\quad
\text{the section chooses }u.
\]

A section \(\varepsilon\) is admissible if it is upward closed with respect to inclusion of half-spaces: whenever $h_u^+\subset h_v^+
$ and $\varepsilon(u)=+1,$ then
$
\varepsilon(v)=+1.
$

Define \(\mathcal M(\mathbb{I}^n)\) to be the set of measurable admissible sections
$
\varepsilon:\mathrm{dS}^n\to\{\pm1\},
$
such that
$
\mu_{\mathrm{dS}^n}\{u\in\mathrm{dS}^n:\varepsilon(u)\neq \varepsilon_o(u)\}<\infty.
$
Define
\[\begin{aligned}
d_{\text{Section}}(\varepsilon,\delta)
=2\mu_{\mathcal{H}}(\varepsilon\triangle\delta)
&=2\int_{\mathrm{dS}^n}
\mathbf 1_{\varepsilon(u)\neq\delta(u)}
\,d\mu_{\mathrm{dS}^n}(u)\\&=
\int_{\mathrm{dS}^n}
\mathbf |\varepsilon(u)-\delta(u)|
\,d\mu_{\mathrm{dS}^n}(u)
\end{aligned}
\]

First, \(d\) is a pseudo-metric. Non-negativity and symmetry are immediate. For the triangle inequality, for almost every \(u\in\mathrm{dS}^n\),
\[
\mathbf 1_{\varepsilon(u)\neq\eta(u)}
\le
\mathbf 1_{\varepsilon(u)\neq\delta(u)}
+
\mathbf 1_{\delta(u)\neq\eta(u)}.
\]
Integrating gives
$
d(\varepsilon,\eta)
\le
d(\varepsilon,\delta)+d(\delta,\eta).
$
Thus \((\mathcal M(\mathbb{I}^n),d)\) is a pseudo-metric space.

We now prove the median property. Let
$
\varepsilon_1,\varepsilon_2,\varepsilon_3\in\mathcal M(\mathbb{I}^n).
$
Define
$
m(u)
=
\operatorname{majority}
\{\varepsilon_1(u),\varepsilon_2(u),\varepsilon_3(u)\}
=
\operatorname{sgn}
\bigl(
\varepsilon_1(u)+\varepsilon_2(u)+\varepsilon_3(u)
\bigr).
$

We check that \(m\in\mathcal M(\mathbb{I}^n)\). Let
$
D_i
=
\{u\in\mathrm{dS}^n:\varepsilon_i(u)\neq\varepsilon_o(u)\}.
$
Each \(D_i\) has finite measure. If
$
m(u)\neq\varepsilon_o(u),
$
then at least one of the three values \(\varepsilon_i(u)\) differs from \(\varepsilon_o(u)\). Therefore
$
\{u:m(u)\neq\varepsilon_o(u)\}
\subset
D_1\cup D_2\cup D_3.
$
So
$
\mu\{u:m(u)\neq\varepsilon_o(u)\}<\infty.
$ And admissibility is obvious, so
$
m\in\mathcal M(\mathbb{I}^n).
$

We now show that \(m\) is a median point of
\(\varepsilon_1,\varepsilon_2,\varepsilon_3\). For each pair
\(i,j\in\{1,2,3\}\), and for every \(u\in\mathrm{dS}^n\), we have
$
\mathbf 1_{\varepsilon_i(u)\neq\varepsilon_j(u)}
=
\mathbf 1_{\varepsilon_i(u)\neq m(u)}
+
\mathbf 1_{m(u)\neq\varepsilon_j(u)}.
$
Integration gives
$
d(\varepsilon_i,\varepsilon_j)
=
d(\varepsilon_i,m)+d(m,\varepsilon_j).
$
So
$
m\in I(\varepsilon_i,\varepsilon_j).
$
By symmetry,
$
m\in
I(\varepsilon_1,\varepsilon_2)
\cap
I(\varepsilon_2,\varepsilon_3)
\cap
I(\varepsilon_3,\varepsilon_1).
$
And
$
M(\varepsilon_1,\varepsilon_2,\varepsilon_3)\neq\varnothing.
$

It remains to prove that this median set has diameter zero. Let
$
\eta\in M(\varepsilon_1,\varepsilon_2,\varepsilon_3)
$
be any median point. Then for every pair \(i,j\),
$
d(\varepsilon_i,\varepsilon_j)
=
d(\varepsilon_i,\eta)+d(\eta,\varepsilon_j).
$
Equivalently,
\[
\int_{\mathrm{dS}^n}
\left(
\mathbf 1_{\varepsilon_i\neq\eta}
+
\mathbf 1_{\eta\neq\varepsilon_j}
-
\mathbf 1_{\varepsilon_i\neq\varepsilon_j}
\right)
\,d\mu_{\mathrm{dS}^n}
=
0.
\]
The integrand is pointwise nonnegative, because
$
\mathbf 1_{\varepsilon_i\neq\varepsilon_j}
\le
\mathbf 1_{\varepsilon_i\neq\eta}
+
\mathbf 1_{\eta\neq\varepsilon_j}.
$
Therefore, for each pair \(i,j\),
$
\mathbf 1_{\varepsilon_i\neq\eta}
+
\mathbf 1_{\eta\neq\varepsilon_j}
-
\mathbf 1_{\varepsilon_i\neq\varepsilon_j}
=
0
$
for almost every \(u\).

So for $(i,j)\in\{(1,2),(2,3),(3,1)\}$, we have
\[
\begin{aligned}
\{u\in \mathrm{dS}^n:\eta(u)\ne m(u)\}&\subset\bigcup_{(i,j)}\{u\in \mathrm{dS}^n:\eta(u)\ne m(u)=\varepsilon_i(u)=\varepsilon_j(u)\}\\&\subset\bigcup_{(i,j)}\{u\in \mathrm{dS}^n:\mathbf 1_{\varepsilon_i\neq\eta}
+
\mathbf 1_{\eta\neq\varepsilon_j}
-
\mathbf 1_{\varepsilon_i\neq\varepsilon_j}
\ne
0\}
\end{aligned}
\]

Therefore
$
\eta(u)=m(u)
$
for almost every \(u\).
So
$
\operatorname{diam}
M(\varepsilon_1,\varepsilon_2,\varepsilon_3)
=
0.
$

Therefore, for every triple
$
\varepsilon_1,\varepsilon_2,\varepsilon_3\in\mathcal M(\mathbb{I}^n),
$
the median set is nonempty and has diameter zero. So
$
\mathcal M(\mathbb{I}^n)
$
is a median pseudo-metric space.
\end{proof}

\smallskip
In this proof, We have
$
h_u^- = h_{-u}^+.
$

\begin{prop}
\label{prop:section oddness} 
If
$
\varepsilon:\mathrm{dS}^n\to\{\pm1\}
$
represents an admissible section, we have $\varepsilon$ is an almost odd function. That is,
$
\varepsilon(-u)=-\varepsilon(u)\ a.e.
$
\end{prop}

To see it is almost odd, see lemma \ref{lem:border}.

\subsection{Example: A non-point section in the medianization}
A natural question is whether every point in the medianization
$
\mathcal M(\mathbb{I}^n)
$
is a point section. We describe explicitly a section in the medianization of $I^2$ which is not generated by a single point of \(\mathbb I^2\). This example also gives evidence that the Hausdorff distance between $\mathbb{I}^n$ and $\mathcal M(\mathbb{I}^n)$ is finite.

For the base point
$
o=(1,0,0),
$
we have
$
\langle o,u\rangle=-u_0,
$
and therefore
$
\sigma_o=\{u\in \mathrm{dS}^2:u_0<0\}.
$

Parametrize de Sitter space by
\[
u(r,\phi)
=
(\sinh r,\cosh r\cos\phi,\cosh r\sin\phi),
\qquad
r\in\mathbb R,\quad \phi\in \mathbb R/2\pi\mathbb Z.
\]

Then
$
\sigma_o=\{(r,\phi):r<0\}.
$

Choose three points in \(I^2\) at the same hyperbolic distance \(R>0\) from \(o\), separated by angles \(2\pi/3\):
$
x_i
=
(\cosh R,\sinh R\cos\theta_i,\sinh R\sin\theta_i),
$
where
$
\theta_1=0,\ 
\theta_2=\frac{2\pi}{3},\ 
\theta_3=\frac{4\pi}{3}.
$

For \(u=u(r,\phi)\), we compute
\[
\langle x_i,u\rangle
=
-\cosh R\sinh r
+
\sinh R\cosh r\cos(\phi-\theta_i)
\]

So
$
u\in \sigma_{x_i}
\iff
\tanh r<(\tanh R)\cdot\cos(\phi-\theta_i).
$
Put
$
a=\tanh R.
$
Then
$
\sigma_{x_i}
=
\left\{
(r,\phi):
\tanh r<a\cos(\phi-\theta_i)
\right\}.
$

Now define the median section
$
m=m(\sigma_{x_1},\sigma_{x_2},\sigma_{x_3}).
$
Equivalently,
$
u\in m
$
if and only if \(u\) belongs to at least two of the three sets
$
\sigma_{x_1},\ \sigma_{x_2},\ \sigma_{x_3}.
$

This median section is not generated by a point of \(\mathbb I^2\), because given any point $x$ it is easy to use plane geometry to find a wall to divide $x$ with at least 2 given end points of the triangle. But in this symmetric example it remains at bounded distance from the base section \(\sigma_o\).

Thus
\[
m
=
\left\{
(r,\phi):
\tanh r<aM(\phi)
\right\}
\]
Here
\[
M(\phi)
=
\operatorname{med}
\left(
\cos\phi,\,
\cos\left(\phi-\frac{2\pi}{3}\right),\,
\cos\left(\phi-\frac{4\pi}{3}\right)
\right)
\]

This section is not generated by a single point of \(\mathbb I^2\). Indeed, if
\[
q=(\cosh\rho,\sinh\rho\cos\alpha,\sinh\rho\sin\alpha),
\]
then the point section \(\sigma_q\) has the form
\[
\sigma_q
=
\left\{
(r,\phi):
\tanh r<\tanh\rho\cos(\phi-\alpha)
\right\}.
\]
Thus the boundary of a point section is always a single cosine graph
\[
\tanh r=\tanh\rho\cos(\phi-\alpha).
\]
By contrast, the boundary of \(m\) is
$
\tanh r=aM(\phi).
$
This is a piecewise cosine function, not a single cosine function.

For example, on the sector
$
-\frac{\pi}{3}\le \phi\le \frac{\pi}{3},
$
one has
\[
M(\phi)
=
\begin{cases}
\cos\left(\phi+\frac{2\pi}{3}\right),
&
-\frac{\pi}{3}\le \phi\le 0,\\[4pt]
\cos\left(\phi-\frac{2\pi}{3}\right),
&
0\le \phi\le \frac{\pi}{3}.
\end{cases}
\]
In particular,
$
M(0)=-\frac12,
\ 
M\left(\frac{\pi}{3}\right)=\frac12.
$
So \(m\neq \sigma_o\). Moreover, \(m\) is invariant under rotation by \(2\pi/3\). If \(m=\sigma_q\) for some \(q\in I^2\), then \(q\) would have to be fixed by this rotation, hence \(q=o\), contradicting \(m\neq \sigma_o\). Therefore
$
m\notin \{\sigma_x:x\in I^2\}.
$

Finally, we compute its distance from the base section. The symmetric difference \(\sigma_m\triangle \sigma_o\) is the region between
$
r=0
$
and
$
r=\operatorname{artanh}(aM(\phi)).
$
Therefore
\[
\mu_{\mathrm{dS}^n}(m\triangle \sigma_o)
=
\int_0^{2\pi}
\frac{a|M(\phi)|}{\sqrt{1-a^2M(\phi)^2}}
\,d\phi
\]
Using the above piecewise formula for \(M\), this gives
\[
\mu_{\mathrm{dS}^n}(m\triangle \sigma_o)
=
12\left(
R-\operatorname{arsinh}
\left(
\frac{\sqrt3}{2}\sinh R
\right)
\right)
\]
Up to the positive constant scalar, this is the distance in the medianization.

As \(R\to\infty\), it is bounded:
\[
R-\operatorname{arsinh}
\left(
\frac{\sqrt3}{2}\sinh R
\right)
\longrightarrow
\log\frac{2}{\sqrt3}
\]

\section{Bound of the Hausdorff distance}

\subsection{One-Wall Modifications and Forced Half-Spaces in the Medianization of Hyperbolic Space}
We now study what happens when one reverses the choice of a wall.
The admissibility condition forces additional walls to follow. By inclusion condition for half-spaces, or intersection condition for walls, it is easy to find them.

Consider first
$
u_0=(0,1,0,\dots,0).
$
$
H_{u_0}
=
\{x_1>0\},
\ 
H_{-u_0}
=
\{x_1<0\}.
$
Recall that an admissible section chooses one and only one of $\{u_0,-u_0\}$, by property \ref{prop:section oddness}. Suppose we reverse the choice to $H_{-u_0}$ instead of $H_{u_0}$. Admissibility forces every half-space containing $H_{-u_0}$ to be chosen. Using property \ref{prop:Half-Space Inclusion}, for $v=(\sinh r,\cosh r\,\omega)\in \mathrm{dS}^n$, the forced region is
\[
\left\{
(r,\omega):
r\le0,
\ \omega_1\le-\frac1{\cosh r}
\right\}
\]

Similarly, changing from $-u_0$ to $u_0$ will force region
\[
\left\{
(r,\omega):
r\le0,
\ \omega_1\ge\frac1{\cosh r}
\right\}
\]

Thus for each fixed $r$, the forced walls form a semi-spherical cap in $S^{n-1}$.

However, one may have noticed that the wall represented by
$u_0$
passes through the base point $o$. Consequently the resulting modification changes only a measure-zero subset in $\sigma_{o}$ and produces no positive distance in the medianization.

To obtain a nontrivial change, move the wall away from the base point. Define
$
u_a
=
(\sinh a,\cosh a,0,\dots,0),
\ 
a<0.
$

Then $H_{u_a}$ contains the base point, while $H_{-u_a}$ does not.

In$\sigma_o$, reverse the choice from $H_{u_a}$ to $H_{-u_a}$.
Admissibility forces all half-spaces containing $H_{-u_a}$.
Equivalently, define
$
D_a
=
\{v\in \sigma_o:
H_{-u_a}\subseteq H_v\}.
$

Using inclusion condition for half-spaces
$
u_a-v\in\overline{\mathcal C}^{+},
$
we obtain
\[
\begin{aligned}
D_a
&=
\left\{
(r,\omega):
0< r\le -a,
\ 
\sinh a\,\sinh r
-
\cosh a\,\cosh r\,\omega_1
\ge1
\right\}
\\&=
\left\{
(r,\omega):
0< r\le -a,
\ 
\omega_1
\le
\frac{
-1+\sinh a\,\sinh r
}{
\cosh a\,\cosh r
}
\right\}
\end{aligned}
\]

Recall that an admissible section is an odd function, so the admissible section with  minimal modification is
$
\sigma^{(a)}
=
(\sigma_o\setminus D_a)\cup(-D_a).
$

The invariant measure on de Sitter space is
$
d\mu
=
\cosh^{n-1}r\,dr\,d\omega.
$ Therefore the distance from the modified section to the base section is
\[
d(\sigma^{(a)},\sigma_o)
=
\frac{
n-1
}{
\operatorname{Vol}(S^{n-2})
}
\mu_{\mathrm{dS}^n}(D_a)
\]
\[
\mu_{\mathrm{dS}^n}(D_a)
=
\int_0^{-a}
\cosh^{n-1}r
\operatorname{Area}
\left\{
\omega\in S^{n-1}:
\omega_1
\le
\frac{
-1+\sinh a\,\sinh r
}{
\cosh a\,\cosh r
}
\right\}
dr
\]

\begin{prop}
\[
\mu_{\mathrm{dS}^n}(D_a)
\sim
\frac{
\operatorname{Vol}(S^{n-2})
}{
n(n-1)
}
\,(-a)^n, \text{ for small } a.
\]
\end{prop}

\begin{proof}
Expand near $0$,
\[
1+\frac{-1+\sinh a\sinh r}{\cosh a\cosh r}=\frac{-1+\cosh(-a-r)}{\cosh a\cosh r}\sim\frac{1}{2}(-a-r)^2
\]

For semi-spherical cap that when
$
\omega_1\le -1+\varepsilon,
$
with $\varepsilon$ small, we have
\[
\operatorname{Area\{\cdots\}}
\sim
\frac1{(n-1)}\operatorname{Vol}(S^{n-2})(2\varepsilon)^{(n-1)/2}.
\]

Here
$
\varepsilon
\sim
\frac{1}{2}(-a-r)^2,
$ uniformly,
so
\[
\operatorname{Area\{\cdots\}}
\sim
\frac1{(n-1)}\operatorname{Vol}(S^{n-2})\,(-a-r)^{(n-1)}
\]

On the other hand, for small $r$, we have
\[
\cosh^{n-1}r
\sim
1+\frac{(n-1)}2r^2+o(r^4)
\]

So
\[
\cosh^{n-1}r\,
\operatorname{Area\{\cdots\}}
\sim
\frac1{(n-1)}\operatorname{Vol}(S^{n-2})\,(-a-r)^{(n-1)}
\]

\[
\mu_{\mathrm{dS}^n}(D_a)
=
\int_0^{-a}
\frac1{(n-1)}\operatorname{Vol}(S^{n-2})\,(-a-r)^{(n-1)}
dr
=\frac{\operatorname{Vol}(S^{n-2})}{n(n-1)}(-a)^n
\]
\end{proof}

\begin{prop}
\label{prop:large section change}
\[
\mu_{\mathrm{dS}^n}(D_a)
\sim
\frac{
\operatorname{Vol}(S^{n-2})
}{
n-1
}
\,(-a), \text{ for large } a.
\]
\end{prop}

\begin{proof}

Fix constants $R,L> 0$, $-a>R+L$, and divide the interval $[0,-a]$ into three regions:
$
[0,-a]
=
[0,R]\cup [R,-a-L]\cup[-a-L,-a]
$.

Let $s=\ln(
\frac{
\cosh\frac{-a+r}{2}
}{
\sinh\frac{-a-r}{2}
})=r+
\log(1+e^{a-r})
-
\log(1-e^{-(-a-r)})\ge r$, then 
$
\frac{-1+\sinh a\,\sinh r}
{\cosh a\,\cosh r}=-\tanh s
$.
And For $s\ge0$,
\[
\operatorname{Area}
\{\omega\in S^{n-1}:\omega_1\le-\tanh s\}
=
\operatorname{Vol}(S^{n-2})\int_s^{+\infty} \frac{1}{\cosh^{n-1}t} \,dt.
\]

\smallskip
For $[0,R]$, 
\[
\begin{aligned}
0
&\le
\int_0^R
\cosh^{n-1} r\,
\operatorname{Area}
\{\omega_1\le -\tanh s\}
\,dr
\\
&\le
R\cosh^{n-1} R\,
\operatorname{Vol}(S^{n-1})
\end{aligned}
\]
The finite upper bound is independent of $a$.

\smallskip
For $[R,-a-L]$, 
$
-a-r\ge L.
$

Using
$
u_A(r)-r
=
\log(1+e^{a-r})
-
\log(1-e^{-(-a-r)}),
$
we estimate
\[
0\le \log(1+e^{a-r})\le e^{a-r}\le e^{-(-a-r)}.
\]
\[
0
\le
-\log(1-e^{-(-a-r)})
\le
\frac{e^{-(-a-r)}}{1-e^{-(-a-r)}}
\le
\frac{e^{-(-a-r)}}{1-e^{-L}}.
\]

So
$
0\le s-r
\le
(1+\frac{1}{1-e^{-L}})\cdot e^{-(-a-r)}
$

Thus
$
s
=
r+O\bigl(e^{-(-a-r)}\bigr)
$, $R\le r\le -a-L$.

By
$
\frac{1}{\cosh^{n-1}t} 
=
2^{n-1} e^{-{(n-1)}t}
\bigl(1+O(e^{-2t})\bigr),
$
uniformly for $t\ge R$, integration from $s$ to $+\infty$ gives for $R\le r\le -a-L$
\[
\int_s^{+\infty}\frac1{\cosh^{n-1}t}dt=\frac{2^{n-1}}{n-1}e^{-{(n-1)}s}(1+O(e^{-2r}))
\]
\[
\cosh^{n-1}r\int_s^{+\infty}\frac1{\cosh^{n-1}t}dt=
\frac1{n-1}+O(e^{-2r})+O(e^{-(-a-r)})
\]
So
\[\begin{aligned}&\operatorname{Vol}(S^{n-2})\int_R^{-a-L}\cosh^{n-1}r\int_s^{+\infty}\frac1{\cosh^{n-1}t}dtdr\\=&\frac{\operatorname{Vol}(S^{n-2})}{n-1}(-a-L-R)+O(\int_R^{-a-L}e^{-2r})+O(\int_R^{-a-L}e^{-(-a-r)})\end{aligned}\]

We have $\int_R^{(-a-L)}e^{-2r} dr\le\frac12e^{-2R}$ and $\int_R^{(-a-L)}e^{-(-a-r)}dr\le e^{-L}$, so

\[\begin{aligned}\operatorname{Vol}(S^{n-2})\int_R^{-a-L}\cosh^{n-1}r\int_s^{+\infty}\frac1{\cosh^{n-1}t}dtdr=\frac{V_{n-1}}{{n-1}}(-a)+O(1)\end{aligned}\]

\smallskip
For $[-a-L,-a]$, as $r\to -a$, $-a-r\to0$. Expand at $r=-a$, we have
\[
\begin{aligned}
1-\tanh s&=
\frac{\cosh(-a-r)-1}
{\cosh a\,\cosh (-a-(-a-r))}\\&=\frac{\cosh(-a-r)-1}{\frac12e^{-a}(1+e^{2a})\cdot
\frac12e^{-a-(-a-r)}
\left(1+e^{2a+2(-a-r)}\right)}
\\&=
\frac{4e^{2a+(-a-r)}(\cosh(-a-r)-1)}{(1+e^{2a+2(-a-r)})(1+e^{2a})}
\end{aligned}
\]

As $-a\to+\infty$, we have
\[\begin{aligned}1-\tanh s
&\sim
4e^{2a+(-a-r)}(\cosh(-a-r)-1)(1+o(1))
\end{aligned}\]

We have uniforly for $-a-L\le r\le -a$,
\[
\sup_{-a-L\le r\le -a}(1-\tanh s)\to0
\qquad
(-a\to+\infty),
\]

So inspect the small cap, we have for $-a-L\le r\le -a$,
\[
\begin{aligned}
&\operatorname{Area}
\{\cdots\}
\\
&\qquad
\sim
\frac{\operatorname{Vol}(S^{n-2})}{{n-1}}
\left(
8e^{a-r}(\cosh (-a-r)-1)
\right)^{\frac{n-1}2}
(1+o(1))^{\frac{n-1}2}
\\&\qquad\sim
\frac{\operatorname{Vol}(S^{n-2})}{n-1}
\left(
8e^{a-r}(\cosh (-a-r)-1)
\right)^{\frac{n-1}2}
\end{aligned}
\]

On the other hand,
\[
\cosh^{n-1}(-a-(-a-r))
=
2^{-({n-1})}e^{({n-1})(-a-(-a-r))}(1+o(1))
\]
uniformly for $-a-L\le r\le -a$.

So
\[
\begin{aligned}
&\cosh^{n-1}r
\operatorname{Area}
\{\omega_1\le -\tanh s\}
\\&\qquad
\sim
\frac{\operatorname{Vol}(S^{n-2})}{n-1}
\left(
2e^{-(-a-r)}(\cosh (-a-r)-1)
\right)^{\frac{n-1}2}
(1+o(1))
\\&\qquad\sim
\frac{\operatorname{Vol}(S^{n-2})}{n-1}
\left(
2e^{-(-a-r)}(\cosh (-a-r)-1)
\right)^{\frac{n-1}2}\\&\qquad=\frac{\operatorname{Vol}(S^{n-2})}{n-1}
(1-e^{-(-a-r)})^{n-1}
\end{aligned}
\]

This is uniforly for $-a-L< r\le -a$, extending continuously to $r=-a-L$. Therefore
$
\cosh^{n-1}r
\operatorname{Area}
\{\omega_1\le -\tanh s\}\le C_L
$
for some constant $C_L$ depends only on $L$. And  \[
\int_{-a-L}^{-a}
\cosh^{n-1} r\,
\operatorname{Area}
\{\cdots\}
\,dr
\le LC_L,
\]
The finite upper bound is independent of $a$.

Adding them together, we have
\[
\mu_{\mathrm{dS}^n}(D_a)
=
\frac{\operatorname{Vol}(S^{n-2})}{n-1}(-a)
+O(1)
\qquad
(a\to-\infty).
\]
\end{proof}

We see that the total forced-wall measure grows linearly in \(-a\). That is because when $-a$ large, the hyperplanes that separates $o$ and $u_a$ are almost the hyperplanes that separates $o$ and $(\cosh a,\sinh a,0,\cdots,0)$, which is at distance $-a$ from $o$. The hyperplane bounded by $u_a$ is orthogonal to the geodesic line from $o$ to that point, as we will discuss after.

Moreover, by Tonelli's theorem, we can change the order of integration of $\mu(D_a)$.

\subsection{Border functions of admissible sections}

Fix the base point
$
o=(1,0,\dots,0)\in \mathbb{I}^n.
$
The space of directions at \(o\) is
$
S^{n-1}.
$
For every direction
$
\omega\in S^{n-1},
$ 
let
$
\gamma_\omega(t)
=
(\cosh t,\sinh t\,\omega)
$
be the geodesic through \(o\) in direction \(\omega\).

\begin{prop}
The hyperplane orthogonal to this geodesic at the point
$
\gamma_\omega(t)
$
is bounded by $u(t,\omega)$ and $-u(t,\omega)=u(-t,-\omega)$, where
\[
u(t,\omega)
=
(\sinh t,\cosh t\,\omega)
\in \mathrm{dS}^n.
\]
\end{prop}
\begin{proof}
Let
$
\gamma_\omega(t)=(\cosh t,\sinh t\,\omega)
$
be a geodesic in \(\mathbb{I}^n\). A hyperbolic hyperplane \(P\subset \mathbb{I}^n\) is orthogonal to \(\gamma_\omega\) at \(\gamma_\omega(t)\) if:

\begin{enumerate}
\item \(\gamma_\omega(t)\in P\);
\item the tangent vector \(\gamma_\omega'(t)\) is orthogonal, with respect to the hyperbolic metric on \(\mathbb{I}^n\), to every tangent vector of \(P\) at \(\gamma_\omega(t)\).
\end{enumerate}

Now define
$
u(t,\omega)=(\sinh t,\cosh t\,\omega).
$
The associated hyperplane is
\[
P_{u(t,\omega)}
=
\{x\in \mathbb{I}^n:\langle x,u(t,\omega)\rangle=0\}.
\]

First,
$
\gamma_\omega'(t)
=
(\sinh t,\cosh t\,\omega)
=
u(t,\omega).
$

Also,
$
\langle u(t,\omega),u(t,\omega)\rangle
=
-(\sinh t)^2+(\cosh t)^2|\omega|^2
=
1,
$
so \(u(t,\omega)\) lies in $\mathrm{dS}^n$ and corresponds to a hyperbolic hyperplane.

Now verify that $
\gamma_\omega(t)\in P_{u(t,\omega)}.
$:
\[
\begin{aligned}
\langle \gamma_\omega(t),u(t,\omega)\rangle
&=
-\cosh t\sinh t
+\sinh t\cosh t\,\langle \omega,\omega\rangle_{\mathbb R^n}  \\
&=
-\cosh t\sinh t+\sinh t\cosh t=0.
\end{aligned}
\]

Next, let
$
v\in T_{\gamma_\omega(t)}P_{u(t,\omega)}.
$
Since
$
P_{u(t,\omega)}
=
\mathbb{I}^n\cap u(t,\omega)^\perp,
$
we have
$
v\in u(t,\omega)^\perp.
$
Therefore
$
\langle v,u(t,\omega)\rangle=0.
$
But \(u(t,\omega)=\gamma_\omega'(t)\), so
$
\langle v,\gamma_\omega'(t)\rangle=0.
$

The hyperbolic metric on \(\mathbb{I}^n\) is the restriction of the Lorentz form to tangent spaces. Therefore
$
g_{\mathbb H^n}(v,\gamma_\omega'(t))
=
\langle v,\gamma_\omega'(t)\rangle
=
0.
$
Thus every tangent vector to \(P_{u(t,\omega)}\) at \(\gamma_\omega(t)\) is orthogonal to \(\gamma_\omega'(t)\). So
$
P_{u(t,\omega)}
$
is exactly the hyperbolic hyperplane orthogonal to \(\gamma_\omega\) at \(\gamma_\omega(t)\).
\end{proof}

Let
$
\tau\in\mathcal M(\mathbb{I}^n)
$
be an admissible section, and let
$
\sigma_\tau\subseteq\mathrm{dS}^n
$
denote the corresponding set of oriented walls whose positive
half-space is chosen by \(\tau\).

Given a base point $o$, for any other point $x$, int the following lemma, the wall orthogonal to the geodesic line between $o$ and $x$ and containing $x$ is noted as $W_x$.
\begin{lem} 
\label{lem:border}
Let \(n \geq 2\). For every admissible section \(\tau \in \mathcal{M}(\mathbb{H}^n)\) and non-trivial geodesic \([a,x_0]\) in \(\mathbb{H}^n\), one of the following cases occurs: \begin{enumerate} \item there exists \(\theta \in (a,x_0)\) such that \(\tau\) coincides with \(\sigma_a\) on every \(W_x\) with \(x\in (\theta,x_0]\) and \(\tau\) coincides with \(\sigma_{x_0}\) on every \(W_x\) with \(x\in [a,\theta)\); \item \(\tau\) coincides with \(\sigma_a\) on every \(W_x\) with \(x\in (a,x_0)\); \item \(\tau\) coincides with \(\sigma_{x_0}\) on every \(W_x\) with \(x\in (a,x_0)\). \end{enumerate} \end{lem}

See Lemma 5.3 of ~\cite{WALL}.
Fix a based point $x_0$ above, then for every fixed direction
$
\omega\in S^{n-1},
$ there is a oriented geodesic line going through $x_0$. And given an admissible section $\tau$,
there exists a unique threshold value
\[
b_\tau(\omega)\in\mathbb R\cup\{\pm\infty\}
\] divides the choice. $+\infty$ means case (2) for any $a$ from $x_0$ in $\omega$ direction. $-\infty$ means case (2) for any $a$ from $x_0$ in $-\omega$ direction. Positive and finite means $\theta$ from $x_0$ in $\omega$ direction and $b_r(\omega)$ distance for far $a$ in $\omega$ direction. Negative and finite means $\theta$ from $x_0$ in $-\omega$ direction and $-b_r(\omega)$ distance for far $a$ in $-\omega$ direction. Zero when (3) for $a$ in $\omega$ and $-\omega$ directions.

However, for $\mathcal M(\mathbb{I}^n)$, $\pm\infty$ is impossible, otherwise the section will be at an infinite distance from the point section, see prop \ref{prop:large section change} with $a\to-\infty$.

Thus every admissible section in $\mathcal M(\mathbb{I}^n)$ determines a \textbf{border function}
\[
b_\tau:S^{n-1}\to\mathbb R.
\]

Note that a border function replies on the admissible section and the choice of base point. Since
$
-u(t,\omega)=u(-t,-\omega),
$
and that a section chooses exactly one orientation of every wall, we get oddness for border function
\[
b_\tau(-\omega)
=
-b_\tau(\omega).
\]

For example, for the base point section
$
\sigma_o
=
\{u\in\mathrm{dS}^n:
\langle o,u\rangle>0
\},
$
we simply have
$
b_o(\omega)=0.
$

\subsection{Admissible Sections and the $1$-Lipschitz Condition}
The funtion $b$ cannot be an arbitrary odd function. We will see that it is 1-Lipschitz.
Fix the base point $o$ in \(\mathbb{I}^n\). Let
$
b:S^{n-1}\to\mathbb R
$
be a border function decided by an admissible section $\sigma_b$, and it parameterizes half-spaces which $\sigma_b$ chooses into
$
\sigma_b
=
\{u(r,\omega):r>b(\omega)\,,\omega\in S^{n-1}\},
$

We recall the notation in the modified inclusion condition for half-spaces. Let
$
\beta(\omega)
=
\arcsin(\tanh b(\omega)),
$
and
$
\rho
=
\arcsin(\tanh r)\in (-\frac\pi2,\frac\pi2).
$
Then
$
\sigma_b
=
\{u(\rho,\omega):\rho>\beta(\omega)\}.
$

The section \(\sigma_b\) is admissible if and only if it is upward closed:
whenever
$
u(\rho,\omega)\in \sigma_b
$
and
$
H_{u(\rho,\omega)}
\subseteq
H_{u(\rho',\eta)},
$
then
$
u(\rho',\eta)\in \sigma_b.
$

Using the modified inclusion condition for half-spaces, this becomes:

Ff
$
\rho>\beta(\omega)
$
and
$
\rho'-\rho
\ge
d_{S^{n-1}}(\omega,\eta),
$
then
$
\rho'>\beta(\eta).
$

This is equivalent to
$
\beta(\eta)-\beta(\omega)
\le
d_{S^{n-1}}(\omega,\eta)
$
for all
$
\omega,\eta\in S^{n-1}.
$

Exchanging \(\omega\) and \(\eta\) gives
$
\beta(\omega)-\beta(\eta)
\le
d_{S^{n-1}}(\omega,\eta).
$

So
$
|\beta(\omega)-\beta(\eta)|
\le
d_{S^{n-1}}(\omega,\eta).
$

Finally, since $b$ is an odd function, we have that $\beta$ is also an odd function.

Vice versa. Given a such odd and 1-Lipschitz function, we can see that the inclusion relation of half-spaces holds. Note that an admissible section can change its choice on the zero-measure set $u(b(\omega),\omega)$ and keep admissible.

\begin{thm}
Suppose $b:S^{n-1}\to\mathbb{R}\cup\{\pm\infty\}$ is an odd function, and $\sigma_b$ is the corresponding section, then we have the two conditions equivalent:\\
(1) $\sigma_b$  is admissible.\\
(2) $\beta=\arcsin(\tanh b)$  is $1$-Lipschitz and odd, and $b\ne \pm\infty\ $
\end{thm}
We call $\beta$ a modified border function. So every such modified finite $1$-Lipschitz function corresponds to an admissible section.
\begin{cor}
The modified border function with base point $o$ of an admissible section in $\mathcal M(\mathbb{I}^n)$ is 1-Lipschitz, up to the choice of the section on a zero-measure set.
\end{cor}

\subsection{Independence of the Border Function of the base point}
We can change the base point for the border function, fixing a given admissible section $\tau$. This is easier to see in open ball model. For any hyperbolic point $x$, and any ideal point $X$, there is a unique geodesic line passing through x and tends to $X$. This geodesic line depends on $X$ continuously, and parameterize the direction from $x$ as a point of $S^{n-1}$. We consider the border function $b_x$, and take $y(\omega)$ as the point we get from walking from $x$, along the geodesic line $l(\omega)$ corresponding to $\omega \in S^{n-1}$, and signed distance $b_x(\omega)$. We have $y(\omega) = y(-\omega)$ and $l(\omega) = l(-\omega)$, and note the unique wall orthogonal to $l(\omega)$ at $y(\omega)$ as $w(\omega)$ and we have the border wall set 
\[
W_{x,\tau}^{border}=\{w(\omega)\  |\ \omega\in S^{n-1}  \}. 
\]
\begin{figure}[htbp]
  \centering
  \includegraphics[width=0.6\textwidth]{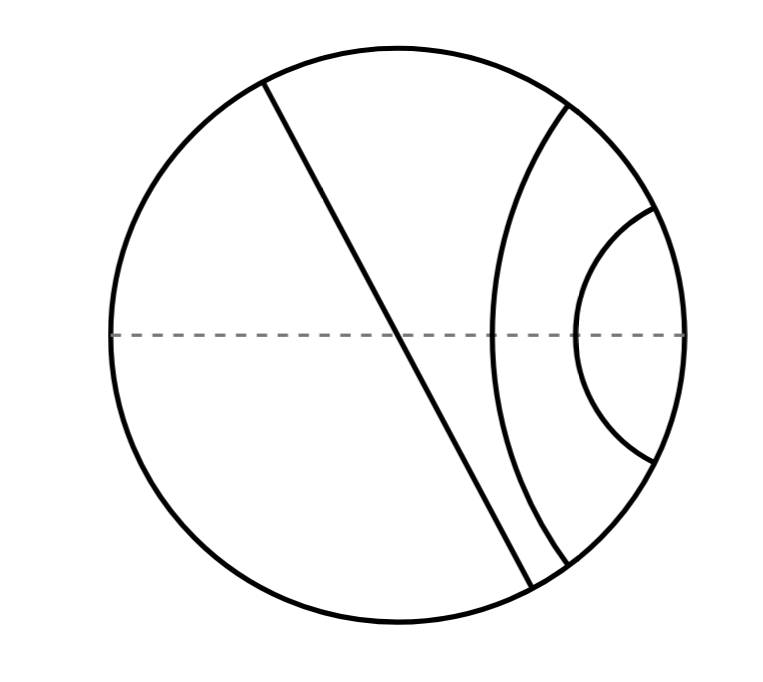}
  \caption{Intersecting one more wall}
  \label{Intersecting one more wall}
\end{figure}
The border wall set is independent of the base point. This is easy to see in the open ball model. Suppose we have another base point $x'$, and corresponding $b_{x'}$, $y'$, $l'$, and $w'$. Take a wall $w(\omega)$, and note the direction in which the geodesic line passing through $x'$ and orthogonal to the wall is $\omega'$. First we consider the geodesic line $l'(\omega')$passing through $x'$. If the chosen half-space of $w(\omega)$ is smaller than that of $w'(\omega')$, then we can intersect a wall orthogonal to $l(\omega)$ between $w(\omega)$ and $w'(\omega')$, and the chooice on the new wall will be conflict with $b_x(\omega)$. And vice versa.

Moreover, we have \textbf{lemma \ref{lem:section distance and border function}}
\[d(\tau,\sigma_x)=\mu_{\mathrm{dS}^n}(W^{border}_{\tau}\triangle \sigma_x)=\int_{S^{n-1}}
\int_0^{|b_x(\omega)|}
\cosh^{n-1}r
\,dr\,d\omega\]
where the $\sigma_x$ also notes the half-spaces containing $x$ and it depends on $x$ continuously. As a result, $d(\tau,\sigma_x)$ is a continuous function of $x$.
Moreover, by isometric actions, we get
\begin{cor}
The modified border function for any base point is $1$-Lipschitz.
\end{cor}
Then \textbf{theorem \ref{thm: border correspondence lip}} is obvious.

We give an example here. See section 2.4. For base point $o=(1,0,\cdots,0)$ and $x_t=(\cosh t,\sinh t,0,\cdots,0), t>0$, the border function \[b_{\sigma_{x_t}}(\omega)=\operatorname{arctanh}(\omega_1\tanh t )\]

\subsection{Distance to a Fixed Admissible Section Along a Geodesic}

Let
$
\tau\in\mathcal M(\mathbb{I}^n)
$
be a fixed admissible section, and let
$
o\in \mathbb{I}^n
$
be a base point. Suppose
$
C
=
d_{\mathcal M}(\tau,\sigma_o)
<
\infty.
$
Let
$
\gamma:[0,\infty)\to \mathbb{I}^n
$
be a geodesic ray starting at \(o\), and write
$
p_R=\gamma(R).
$
Then
$
d_{\mathbb{I}^n}(o,p_R)=R.
$

Since the embedding
$
\mathbb{I}^n\hookrightarrow\mathcal M(\mathbb{I}^n)
$
is isometric,
$
d_{\mathcal M}(\sigma_o,\sigma_{p_R})
=
R.
$

Applying the reverse triangle inequality gives
\[
d_{\mathcal M}(\tau,\sigma_{p_R})
\ge
d_{\mathcal M}(\sigma_o,\sigma_{p_R})
-
d_{\mathcal M}(\tau,\sigma_o).
\]

So
$
d_{\mathcal M}(\tau,\sigma_{p_R})
\ge
R-C
$.
Therefore
$
d_{\mathcal M}(\tau,\sigma_{p_R})
\longrightarrow
\infty
\quad
(R\to\infty).
$

Plus, $min_\gamma\{d_{\mathcal M}(\tau,\sigma_{p_R})\}\ge R-C$ the lower bound depends only on $R$.

\subsection{Minimizing Point Sections and the Locally Minimum Condition}
We study what will happen when the base point is to make the distance locally minimum. We suppose $n\ge2$.

Fix
$
\tau\in \mathcal M(\mathbb{I}^n)
$
, and define
$
D(p)
=
d_{\mathcal M}(\tau,\sigma_p),
\ 
p\in \mathbb{I}^n.
$ We proved that $D$ is a continuous function on $\mathbb{I}^n$ and tends to infinity outside a big enough circle centered at a base point, so \(D\) attains a minimum at some point
$
p_0\in \mathbb{I}^n.
$

Choose de Sitter coordinates centered at \(p_0\):
\[
u(r,\omega)
=
(\sinh r,\cosh r\,\omega),
\qquad
\omega\in S^{n-1}.
\]

This is possible because the isometric action on hyperbolic space also keeps orthogonal relation, so that $D$ is invariant under the action of isometric action.

Let
$
b=b_{p_0}.
$ Notice that we can assume $b\ne 0\ a.e.$ to make the integration only change a bit, just if it satisfies the 1-Lipschitz condition for $\beta$. The discussion section will provide details about it. Briefly speaking, we divide an open disk where $\beta$ is zero into $2^n$ parts, and adjust $\beta$ on it.

The point section \(\sigma_{p_0}\) corresponds to the case
$
r=0.
$ 
Therefore
\[
D(p_0)
=
\int_{S^{n-1}}
\int_0^{|b(\omega)|}
\cosh^{n-1}r\,dr
\,d\omega
=
\int_{S^{n-1}}
\int_0^{|\operatorname{arctanh} \sin \beta(\omega)|}
\cosh^{n-1}r\,dr
\,d\omega
\]

Now move the base point from \(p_0\) in a tangent direction
$
\xi\in S^{n-1}.
$
\begin{lem}
\label{lem:border function base point moved}
Let \(o=(1,0,\dots,0)\in \mathbb{I}^n\),and \(\xi\in S^{n-1}\), and move the base point a distance \(s\) in direction \(\xi\):
\[
p_s=(\cosh s,\sinh s\,\xi).
\]
Let \(b_s\) be the border function of the same section \(\tau\) with respect to the base point \(p_s\), using the frame obtained by parallel transport from \(o\) to \(p_s\). If \(b\) is differentiable at \(\omega\)(by 1-Lipschitz, we know this holds almost everywhere), then
\[
b_s(\omega)
=
b(\omega)
-
s(\xi\cdot\omega)
+
s\tanh b(\omega)
\left\langle
\nabla_S b(\omega),
\xi
\right\rangle
+
o(s).
\]
\end{lem}

\begin{proof}
Let \(g_s\in SO(n,1)\) be the hyperbolic translation sending \(o\) to
\[
p_s=(\cosh s,\sinh s\,\xi).
\]
For \(z=(z_0,\mathbf{z})\in \mathbb R^{n+1}\), write
$
\mathbf{z}
=
z_\perp+(\mathbf{z}\cdot\xi)\xi,
\ 
z_\perp\perp \xi.
$
Then
\[
g_s(z_0,\mathbf{z})
=
\left(
\cosh s\,z_0+\sinh s\,(\mathbf{z}\cdot\xi),
\;
z_\perp+
(\sinh s\,z_0+\cosh s\,(\mathbf{z}\cdot\xi))\xi
\right).
\]

The de Sitter parametrization centered at \(p_s\) is therefore
\[
u_s(\rho,\omega)
=
g_su(\rho,\omega).
\]
Let
$
\alpha=\xi\cdot\omega.
$
Since
$
u(\rho,\omega)=(\sinh\rho,\cosh\rho\,\omega),
$
we obtain
\[
u_s(\rho,\omega)_0
=
\cosh s\,\sinh\rho
+
\sinh s\,\cosh\rho\,\alpha,
\]
and noting sp(space) the last $n$ coordinates
\[
u_s(\rho,\omega)_{\mathrm{sp}}
=
\cosh\rho(\omega-\alpha\xi)
+
\left(
\sinh s\,\sinh\rho
+
\cosh s\,\cosh\rho\,\alpha
\right)\xi.
\]

Now rewrite the same open half-space in the original coordinates centered at \(o\):
\[
u_s(\rho,\omega)=u(r,\eta)
=
(\sinh r,\cosh r\,\eta).
\]
We have
$
\sinh r
=
\cosh s\,\sinh\rho
+
\sinh s\,\cosh\rho\,\alpha.
$
By
$
\cosh s=1+O(s^2),
\ 
\sinh s=s+O(s^3),
$
we have
$
\sinh r
=
\sinh\rho+s\cosh\rho\,\alpha+O(s^2).
$
So
$
r=\rho+s\alpha+O(s^2).
$
Therefore
$
r
=
\rho+s(\xi\cdot\omega)+O(s^2).
$

Next, compare the coordinates. We have
$
\cosh r\,\eta
=
\cosh\rho\,\omega
+
s\sinh\rho\,\xi
+
O(s^2).
$

Since
$
r=\rho+s\alpha+O(s^2),
$
we have
$
\cosh r
=
\cosh\rho+s\alpha\sinh\rho+O(s^2).
$
Thus
\[
\eta
=
\frac{
\cosh\rho\,\omega+s\sinh\rho\,\xi+O(s^2)
}{
\cosh\rho+s\alpha\sinh\rho+O(s^2)
}
=
\frac{
\omega+s\tanh\rho\,\xi+O(s^2)
}{
1+s\alpha\tanh\rho+O(s^2)
}.
\]
Therefore
\[
\eta
=
\omega
+
s\tanh\rho(\xi-\alpha\omega)
+
O(s^2)
=
\omega
+
s\tanh\rho
\bigl(\xi-(\xi\cdot\omega)\omega\bigr)
+
O(s^2)
\]

Now let
$
\rho=b_s(\omega).
$
By definition, the hyperplane bounding by
$
u_s(b_s(\omega),\omega)
$
lies on the fixed border of \(\tau\). Written in the old coordinates, this means
$
r=b(\eta)
$.
Using the expansions above, this becomes
\[
b_s(\omega)+s(\xi\cdot\omega)
=
b\left(
\eta
\right)+O(s^2)
\]

Since
$
b_s(\omega)=b(\omega)+O(s),
$
we have 
\[
\exp_\omega{(s\tanh b(\omega)
\bigl(\xi-(\xi\cdot\omega)\omega\bigr))}=\omega+s\tanh b(\omega)
\bigl(\xi-(\xi\cdot\omega)\omega\bigr)+O(s^2)=\eta+O(s^2)
\]

$b$ is Lipschitz, so $b(\eta)=b(\exp_\omega{(s\tanh b(\omega)
\bigl(\xi-(\xi\cdot\omega)\omega\bigr))})+O(s^2)$

Differential at the right hand side, we have
\[\begin{aligned}
&b(\exp_\omega{(s\tanh b(\omega)
\bigl(\xi-(\xi\cdot\omega)\omega\bigr))})\\&\ =b(\omega)+s\tanh b(\omega)
\left\langle
\nabla_S b(\omega),
\xi-(\xi\cdot\omega)\omega
\right\rangle
+
o(s)
\end{aligned}
\]

Plus that $\langle\nabla _Sb(\omega),\omega\rangle=0$, so

\[
b_s(\omega)
=
b(\omega)
-
s(\xi\cdot\omega)
+
s\tanh b(\omega)
\left\langle
\nabla_S b(\omega),
\xi
\right\rangle
+
o(s).
\]
\end{proof}

Now we prove lemma \ref{lem:locally minimum condition} from lemma \ref{lem:border function base point moved}. One can find a simpler proof in the discussion section.
\begin{proof}

For every \(\xi\in S^{n-1}\), recall the integration on lemma \ref{lem:section distance and border function}, we do differential and use Lipschitz property of $b$ and boundness of border function to get dominated convergence, then get the following equation and note it as (*):
\[
0=
\int_{S^{n-1}}
\operatorname{sgn}(b)\cosh^{n-1}(b)
\left[
-(\xi\cdot\omega)
+
\tanh(b)\,\nabla_S b\cdot \xi
\right]
\,d\omega .
\]
\[
=
-\int_{S^{n-1}}
\operatorname{sgn}(b)\cosh^{n-1}(b)(\xi\cdot\omega)\,d\omega
+
\int_{S^{n-1}}
\operatorname{sgn}(b)\cosh^{n-1}(b)\tanh(b)\,
\nabla_S b\cdot\xi\,d\omega .
\]

Now define
$
Q(t)=\frac{\operatorname{sgn}(t)}{{n-1}}\bigl(\cosh^{n-1} t-1\bigr).
$
$
Q'(t)
=
\operatorname{sgn}(t)\cosh^{n-1}(t)\tanh(t).
$

So
$
\operatorname{sgn}(b)\cosh^{n-1}(b)\tanh(b)\,\nabla_S b
=
\nabla_S(Q\circ b)
$

\[
\int_{S^{n-1}}
\operatorname{sgn}(b)\cosh^{n-1}(b)\tanh(b)\,
\nabla_S b\cdot\xi\,d\omega
=
\int_{S^{n-1}}
\nabla_S(Q\circ b)\cdot\xi\,d\omega
\]

Since \(\nabla_S(Q\circ b)\) is tangent to \(S^{n-1}\), we may replace \(\xi\) by the tangential projection
$
\xi-(\xi\cdot\omega)\omega
=
\nabla_S(\xi\cdot\omega).
$
So
\[
\int_{S^{n-1}}
\nabla_S(Q\circ b)\cdot\xi\,d\omega
=
\int_{S^{n-1}}
\nabla_S(Q\circ b)\cdot\nabla_S(\xi\cdot\omega)\,d\omega
\]
By integration by parts on \(S^{n-1}\),
\[
\int_{S^{n-1}}
\nabla_S(Q\circ b)\cdot\nabla_S(\xi\cdot\omega)\,d\omega
=
-\int_{S^{n-1}}
(Q\circ b)\,\Delta_S(\xi\cdot\omega)\,d\omega
\]
Since for the Laplace operator we have $
\Delta_S(\xi\cdot\omega)=-(n-1)(\xi\cdot\omega),
$ we get
\[
\int_{S^{n-1}}
\nabla_S(Q\circ b)\cdot\xi\,d\omega
=
(n-1)\int_{S^{n-1}}
Q(b)(\xi\cdot\omega)\,d\omega
\]
By the definition of \(Q\), 
$
(n-1)Q(b)
=
\operatorname{sgn}(b)\bigl(\cosh^{n-1}(b)-1\bigr).
$
So
\[
\int_{S^{n-1}}
\operatorname{sgn}(b)\cosh^{n-1}(b)\tanh(b)\,
\nabla_S b\cdot\xi\,d\omega
=
\int_{S^{n-1}}
\operatorname{sgn}(b)
\bigl(\cosh^{n-1}(b)-1\bigr)
(\xi\cdot\omega)\,d\omega
\]

Putting this into (*),
\[
0=
-\int_{S^{n-1}}
\operatorname{sgn}(b)\cosh^{n-1}(b)(\xi\cdot\omega)\,d\omega
+
\int_{S^{n-1}}
\operatorname{sgn}(b)
\bigl(\cosh^{n-1}(b)-1\bigr)
(\xi\cdot\omega)\,d\omega
\]
So
\[
0=
-\int_{S^{n-1}}
\operatorname{sgn}(b)(\xi\cdot\omega)\,d\omega
\]
Since this holds for every \(\xi\in S^{n-1}\), we obtain
\[
\int_{S^{n-1}}
\operatorname{sgn}(b(\omega))\,\omega\,d\omega
=
0
\]
\end{proof}

\subsection{Estimate of the Hausdorff Bound}
We work on \textbf{problem \ref{prob:anal form}} here.

Recall that
$
\beta(\omega)
=
\arcsin(\tanh b(\omega))\in(-\frac\pi2,\frac\pi2).
$
is odd and \(1\)-Lipschitz, and $S^{n-1}$ is compact, we can choose
$
\xi\in S^{n-1}
$
such that

$
\beta(\xi)=\max_{\omega\in S^{n-1}}
\beta(\omega)=:m\in[0,\frac\pi2).
$

By the \(1\)-Lipschitz condition, for any $\omega \in S^{n-1}$,
\[
\begin{cases}
\beta(\omega)>0
\text{ if }
d_{S^{n-1}}(\omega,\xi)<m
\\
\beta(\omega)<0
\text{ if }
d_{S^{n-1}}(\omega,-\xi)<m
\end{cases}
\ \implies\ 
\begin{cases}
\beta(\omega)>0
\text{ if }
\omega\cdot\xi>\cos(m)
\\
\beta(\omega)<0
\text{ if }
\omega\cdot\xi<-\cos(m)
\end{cases}
\]

And when $|\omega\cdot\xi|>\cos(m)$, $sgn(b(\omega))\omega = sgn(b(-\omega))(-\omega)$, so the two semi-sphere caps $\{\omega\ |\ \omega\cdot\xi>\cos(m)\}$ and $\{\omega\ |\ \omega\cdot\xi<-\cos(m)\}$ make the same contribution to the integration in locally minimum condition. The rest strip cannot be too small to cancel it to zero.

Recalling the proof of lemma \ref{lem:locally minimum condition}, we can take inner product on the integration and move $\xi$ in, to consider the component along $\xi$. One can see the reason for being valid is that $\xi$ is fixed and the inner product is linear, if one read the simplier proof.

\[
\begin{aligned}
0
&=
\ \int_{\cos m}^{1}
\int_{S^{n-2}}
\, t
\, d\sigma(\omega')\, dt  \\
&\quad
+
\int_{-1}^{-\cos m}
\int_{S^{n-2}}
\, |t|
\, d\sigma(\omega')\, dt \\
&\quad
+
\int_{-\cos m}^{\cos m}
\int_{S^{n-2}}
\operatorname{sgn}\!\bigl(b(t,\omega')\bigr)
\, t
\, d\sigma(\omega')\, dt .
\end{aligned}
\]
Let
$
t=\omega\cdot\xi.
$
The right-hand side has a lower bound. The cap regions' integration is inrelavent to the sign of $b$ any more; For strip region, use the assumption that the $sgn(b)$ is negative if and only if $t>0$, to make the biggest contribution to negative direction. 
\[
\begin{aligned}
0
&\ge
\int_{\cos m}^{1}
\int_{S^{n-2}}
\, t
\, d\sigma(\omega')\, dt  \\
&\quad
+
\int_{-1}^{-\cos m}
\int_{S^{n-2}}
\, |t|
\, d\sigma(\omega')\, dt \\
&\quad
-
\int_{-\cos m}^{\cos m}
\int_{S^{n-2}}
\, |t|
\, d\sigma(\omega')\, dt \\
\end{aligned}
\]
We cancel the common multiplying factor, and get by symmetry
\[
\int_{\{1>t>\cos(m)\}} t\,d\omega
\le
\int_{\{0<t<\cos(m)\}} t\,d\omega
\]

Using the spherical slicing formula
$
d\omega
=
\operatorname{Vol}(S^{n-2})
(1-t^2)^{\frac{n-3}{2}}
\,dt,
$
the common factor
$
\operatorname{Vol}(S^{n-2})
$
is independent of $t$ so cancels, resulting in
\[
\int_{\cos m}^{1}
t(1-t^2)^{\frac{n-3}{2}}
\,dt
\le
\int_0^{\cos m}
t(1-t^2)^{\frac{n-3}{2}}
\,dt
\]

We have
\[
\int t(1-t^2)^{\frac{n-3}{2}}\,dt
=
-\frac{1}{n-1}
(1-t^2)^{\frac{n-1}{2}}
\]

We get
$
(\sin m)^{n-1}
\le
\frac12
$
Therefore we have the uniform estimate
\begin{lem}
\label{lem:border function upper bound}
When locally minimum condition is satisfied
\[
|b(\omega)|
\le
B_n
:=
\operatorname{artanh}
\left(
2^{-\frac1{n-1}}
\right).
\]
\end{lem}
Consequently, in the half-space measure
\[
\mu_H(\tau,\sigma_{p_0})
=
\int_{S^{n-1}}
\int_0^{|b(\omega)|}
\cosh^{n-1}r
\,dr\,d\omega
\]
just replace $b(\omega)$ by $B_n$, and recall that
\[
d_{\mathcal M{(\mathbb \mathbb{I}^n)}}=
\frac{n-1}{\operatorname{2vol}(S^{n-2})}\mu_{\mathrm{dS}^n}
\]
\begin{lem}
For any section $\tau \in \mathcal{M}(\mathbb{I}^n)$, there is a point $p_0\in \mathbb{I}^n$ minimizing the distance between the section and a point section, and
\[
d_{\mathcal M{(\mathbb \mathbb{I}^n)}}(\tau,\sigma_{p_0})
\le
\frac{n-1}{2}
\frac{\operatorname{Vol}(S^{n-1})}
{\operatorname{Vol}(S^{n-2})}
\int_0^{B_n}
\cosh^{n-1}r
\,dr
\]
\end{lem}

That is,
\[
d_M(\mathbb{I}^n,\mathcal M(\mathbb{I}^n))
\le
\frac{n-1}{2}
\frac{\operatorname{Vol}(S^{n-1})}
{\operatorname{Vol}(S^{n-2})}
\int_0^{
\operatorname{artanh}(2^{-1/(n-1)})
}
\cosh^{n-1}r
\,dr.
\]

We know that 
\[\frac{\operatorname{Vol}(S^{n-1})}{\operatorname{Vol}(S^{n-2})} = \sqrt{\pi} \, \frac{\Gamma\left(\frac{n-1}{2}\right)}{\Gamma\left(\frac{n}{2}\right)}\]

So we prove the theorem \ref{thm:hausdorff bound}

\[
d_M(\mathbb{I}^n,\mathcal M(\mathbb{I}^n))
\le
\sqrt{\pi} 
\frac{n-1}{2}
\, \frac{\Gamma\left(\frac{n-1}{2}\right)}{\Gamma\left(\frac{n}{2}\right)}
\int_0^{
\operatorname{artanh}(2^{-1/(n-1)})
}
\cosh^{n-1}r
\,dr
\]

\subsection{The Hausdorff distance for Dimension 2}

In this chapter, we give an upper bound for $n=2$, then construct a admissible section to show that it is sharp. To prove it is share, we don't calculate the global minimum. Instead, we will show that there is only one local minimum.

When $n=2$, we have $\beta\le\frac\pi6$ by lemma \ref{lem:border function upper bound}. The integration we want to estimate becomes 
$
\int_{S^1}|\tan\beta(\omega)|d\omega
$
And the locally minimum condition becomes
$
\int_{S^1}\text{sgn}( \tan\beta(\omega))\omega d\omega=0
$

There is a $\beta$ satisfying this. For $\omega(\alpha) =(\cos\alpha,\sin\alpha),\,\alpha\in\mathbb{R},$ we set
\[\beta(\omega(\pm\frac{2\pi}3))=\beta(\omega(0))=\frac\pi6\]
\[\beta(\omega(\pm\frac{\pi}3))=\beta(\omega(\pm\pi))=-\frac\pi6\]
and connect these points linearly. We will see that this give the sharp bound of Hausdorff distance.

By the proof of lemma \ref{lem:border function upper bound}, we can see that if $\beta\circ\omega(\alpha)\ne0$ for $\alpha\in (a,b)$, then $b-a\le\frac\pi3$. If we moreover make the interval locally maximal, equivalently $\beta\circ\omega(a)=\beta\circ\omega(b)=0$, then by $1$-Lipschitz property of $\beta$, we have 
$
\beta\circ\omega(\alpha)\le\text{min}\{\alpha-a,b-\alpha\},\,\alpha\in[a,b].
$
So
\[
\int_a^b{|\tan(\beta\circ\omega(\alpha))|}d\alpha\le 2\ln(\cos\frac{b-a}{2})
\]
The subset where $\beta\ne 0$ lying in $S^1$ is of measure at most $2\pi=6\times \frac\pi3$ and of at most countably infinite components, so it is easy to see by convexity that the integration reaches maximal in the case we construct. So for all admissiblt section that is finite-distance away from one of the point section locally minimizing the distance, the distance is at most $12\ln2-6\ln3$. That is,
\[
d_\mathcal M(\mathbb I^2,\mathcal M(\mathbb I^2))
\le
12\ln2-6\ln3
\]

The next thing is to prove that in this construction, the base point is the point globally minimizing the distance. Suppose there is a point $p=(\cosh r,(\sinh r)\,\omega),\,r\ge 0,\,\omega \in S^{1}$, the set of open half-spaces bounding a hyperplane passing $p$ is \[\{u(\alpha)=(\sinh t,(\cosh t)\alpha): t=t(\alpha) = \operatorname{arctanh}( (\tanh r)(\omega \cdot \alpha)),\alpha \in S^{1}\}\]

\begin{thm}
\[
d_\mathcal M(\mathbb I^2,\mathcal M(\mathbb I^2))
=
12\ln2-6\ln3
\]
\end{thm}

\begin{proof}
We note the construction as $\beta_*$, and the corresponding admissible section as $\sigma_*$. Simple calculation shows that \[
\sinh(\operatorname{arctanh}(x))
=
\tan (\arcsin (x))
=
\frac{x}{\sqrt{1-x^2}}
\]
By simple calculation, for any hyperbolic point $p$, we have
\[
d_\mathcal{M}(\sigma_p,\sigma_*)=\int_{-\pi}^{\pi}\,|\tan \beta_* \circ \omega(t)-\sinh \circ \operatorname{arctanh} ((\tanh r)(\alpha \cdot \omega(t)))|\,dt
\]

In fact we don't need to prove it is a global minimum directly. We just need to prove that $p\ne o$ is not a local minimum, because the distance is proved to be tending to infinity when $p$ walks along a geodesic line far away, see prop \ref{prop:large section change}. In the following part we use notation $\beta_*$ instead of $\beta_* \circ \omega$.

\textbf{[Step 1: Transformate it into an easy form.]}

Note $\rho=\tanh r\in(-1,1)$. We get the sign of the things in the absolute value symbol as
$S_{\rho,\theta}:=$
\[
sgn(\tan \beta_* (\alpha)-\frac{\rho \cos(\alpha - \theta)}{(1-\rho^2\cos^2(\alpha-\theta))^{\frac12}} = sgn(\beta_*(\alpha)-\text{arcsin}(\rho\cos(\alpha - \theta)))
\]
Note $q_{\rho,\theta}(\alpha):=\arcsin(\rho\cos(\alpha-\theta))$. By locally minimum condition, we only need to show that, for $p> o$, moving in direction $\theta$,
\[
I_{\rho,\theta}:=\int_0^{2\pi}\,sgn(\beta_*(\alpha)-q_{\rho,\theta}(\alpha))\,\frac{\cos(\alpha-\theta)}{(1-\rho^2\cos^2(\alpha-\theta))^\frac32}d\alpha\ne0
\]
Let $\Phi_\rho(x)=\frac{\sin(x)}{(1-\rho^2)(1-\rho^2\cos^2(x))^\frac12},\,x=\alpha-\theta,$ 
then
$\Phi_\rho^\text{'}(x)=\frac{\cos x}{(1-\rho^2\cos^2(x))^\frac32}$.
\[
I_{\rho,\theta}=2\int_{\Phi_\rho(-\frac\pi2)}^{\Phi_\rho(\frac\pi2)}S_{\rho,\theta}\,d\Phi_\rho(x)
\]
Let $P_{\rho,\theta}=\{\alpha\in[c,c+2\pi):S_{\rho,\theta}>0\}=\{\alpha:\beta(\alpha)>q_{\rho,\theta}(\alpha)\}$, which consists of discrete interval segments for fixed c. We choose c so that  $\beta_*(c)=q_{\rho,\theta}(c)$ and ${\beta_*}_+^\text{'}(c)=1$, then write $\text{interior}(P_{\rho,\theta})=\cup_j\,(a_j,b_j),$ then \[I_{\rho,\theta}=2\Sigma_j(\Phi_\rho(b_j-\theta)-\Phi_\rho(a_j-\theta))\]

This is valid, and we will explain here. First notice that 
\[
q_{\rho,\theta}^\text{'}(\alpha)=-\frac{\rho\sin(\alpha-\theta)}{(1-\rho^2\cos^2(\alpha-\theta))^\frac12}\in(-1,+1)
\]
Plus on every $\frac\pi3$-width interval where $\beta_*$ is increasing or decreasing, $\beta_*$ is linear with slope $\pm 1$. So $\beta_*(\alpha)-q_{\rho,\theta}(\alpha)$ increases(decreases) when and only when $\beta_*(\alpha)$ increases(decreases). That is, for integer $k$, $\beta_*(\alpha)-q_{\rho,\theta}(\alpha)$ increase on $({(2k-1)}\frac\pi3,{2k}\frac\pi3)$ and decrease on $({2k}\frac\pi3,{(2k+1)}\frac\pi3)$. As a result, there is at most one point $\alpha$ satisfying $\beta_*(\alpha)-q_{\rho,\theta}(\alpha)=0$ on such $\frac\pi3$-width closed interval, and the $\alpha$ is sign-changing zero point if and only if it lies in the interior of the interval.

Moreover, we can prove that there is only 3 cases in the period $[c,c+2\pi)$: (1) 2 sign-changing ones; (2) 2 sign-changing ones, 2 sign-keeping ones; (3) 6 sign-changing zero points. First by anti-period $\frac\pi2$ there is a sign-changing zero points, and by anti-period $\frac\pi2$ either it or it plus $\pi$ is the one where $\beta_*$ increases. The one where $\beta_*$ increases is picked as $c$. Next consider the negative peak of the function between $c$ and $c+\pi$. The monotonities are clear, so it is obvious.

\textbf{[Step 2: Get the transformation for end points of segments.]}

We see that the integration we concerned is decided by pairs $(a_j,b_j)$. At these zero points, $1-\rho^2\cos^2(\alpha-\theta)=\cos^2\beta(\alpha)$, we get
\[
\Phi_\rho(\alpha-\theta)=\frac{\sin(\alpha-\theta)}{(1-\rho^2)\cos\beta_*(\alpha)}
\]

Note
\[
c_j=\frac{\sin(b_j-\theta)}{\cos\beta_*(b_j)}-\frac{\sin(a_j-\theta)}{\cos\beta_*(a_j)}
\]
Then $I_{\rho,\theta}=\frac2{1-\rho^2}\Sigma_jc_j,\, j\in J$.

When $\beta_*(\alpha)-q_{\rho,\theta}(\alpha)=0$, we have $\sin\beta_*(\alpha)=\rho\cos(\alpha-\theta)$. On the interval where $\beta_*$ keeps a monotonicity, note $\beta_*(\alpha)=\lambda (\alpha-\gamma),\,\lambda\in\{\pm 1\}$ and $\gamma$ is a constant for the interval. Then
\[
\lambda(\sin\alpha\cos\gamma-\cos\alpha\sin\gamma)=\rho(\cos\alpha\cos\theta+\sin\alpha\sin\theta)
\]
Notice that $\cos\beta(\alpha)>0$. And we get
\[
\sin(\alpha-\theta)(1-\lambda\rho\sin(\theta-\gamma))=(\lambda\rho-\sin(\theta-\gamma))\cos(\alpha-\gamma)
\]
Finally
\[
\frac{\sin(\alpha-\theta)}{\cos(\alpha-\gamma)}=\frac{(\lambda\rho-\sin(\theta-\gamma))}{(1-\lambda\rho\sin(\theta-\gamma))}
\]

For $a_j$, $\lambda=1$. For $b_j$, $\lambda=-1$. We choose $\gamma$ as the mid-point of the interval segment with monotonity. That is, for integer $k$, if $a_j\in[(2k-1)\frac{\pi}{3},2k\frac{\pi}{3}]$, $\gamma_{left}(a_j)={(2k-\frac12)}\frac{\pi}{3}$; if $b_j\in[2k\frac{\pi}{3},(2k+1)\frac{\pi}{3}]$, $\gamma_{right}(b_j)=(2k+\frac{1}{2})\frac{\pi}{3}$. Then $\beta_*(a_j)=a_j-\gamma_{left}(a_j),\,\beta_*(b_j)=\gamma_{right}(b_j)-b_j$. Let $f(x)=\frac{x-\rho}{1-\rho x}$, then
\[
\begin{aligned}
c_j \,
&=
\frac{\sin(b_j-\theta)}{\cos\beta(b_j)}-\frac{\sin(a_j-\theta)}{\cos\beta(a_j)}\\
&=
\frac{-\rho-\sin(\theta-\gamma_{right}(b_j))}{1+\rho\sin(\theta-\gamma_{right}(b_j))}
-
\frac{\rho-\sin(\theta-\gamma_{left}(a_j))}{1-\rho\sin(\theta-\gamma_{left}(a_j))}\\
&=
f(\sin(\gamma_{right}(b_j)-\theta))-f(\sin(\gamma_{left}(a_j)-\theta))
\end{aligned}
\]

Note that we have
\[
f(x)+f(-x)=\frac{x-\rho}{1-\rho x}+\frac{-x-\rho}{1+\rho x}=\frac{2\rho(1-x^2)}{1-\rho^2x^2}<0
\]

One may be worried about if a point is a sign-keeping zeros, so that it lies in two intervals that intersect. In this case we just respect the $(a_j,b_j)$ we are discussing.

Especially, when $a_j$ and $b_j$ lie in two nearby interval segments, we have $\gamma_{right}(b_j)=\gamma_{left}(a_j)+\frac\pi3$. Note $m_j=\gamma_{right}(b_j)-\frac\pi6=\gamma_{left}(a_j)+\frac\pi6$, then we have
\[
\begin{aligned}
c_j \,
&=
f(\sin((m_j-\theta)+\frac\pi6))-f(\sin((m_j-\theta)-\frac\pi6))
\\
&=
\frac {(\rho\cos(m_j-\theta)-\frac12)(\rho-2\cos(m_j-\theta))} {(1-\rho\sin(\frac\pi6+(m_j-\theta)))(1-\rho\sin(\frac\pi6-(m_j-\theta)))}
\end{aligned}
\]

\textbf{[Step 3: Use information of end points for concrete calculation.]}

We recall that there are 3 cases. We discuss respectively.

(1) On $[c,c+2\pi)$, there is 2 zero points of $\beta(\alpha)-q_{\rho,\theta}(\alpha)$.

\begin{figure}[htbp]
  \centering
  \includegraphics[width=0.6\textwidth]{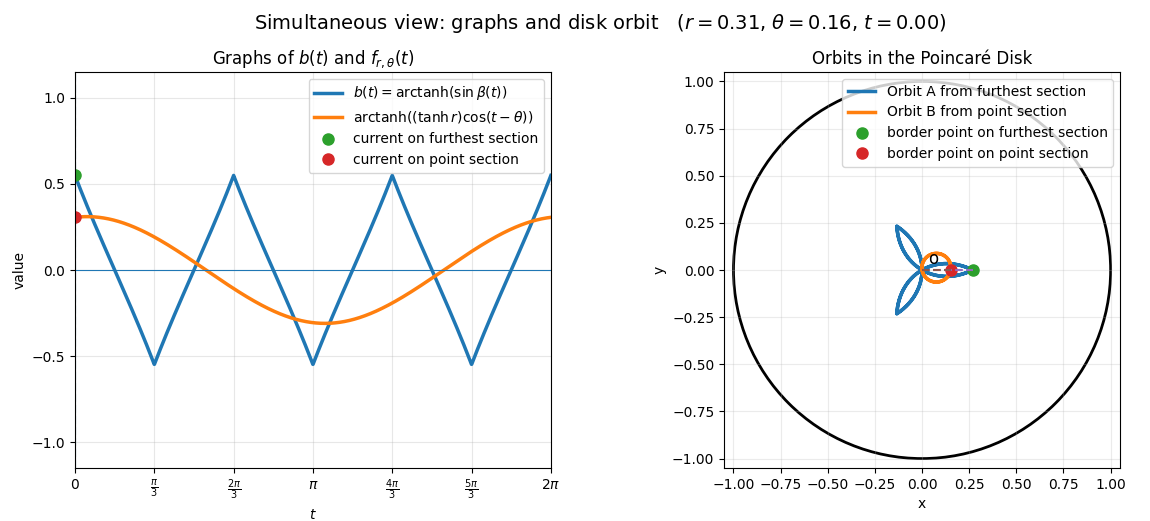}
  \caption{2 zero points}
  \label{2 zero points}
\end{figure}

Then they are all sign-changing zero points, and $\gamma_{right}(b)=\gamma_{left}(a)+\pi$. So
\[
\begin{aligned}
c=&f(\sin(\gamma_{left}(a)+\pi-\theta))+f(\sin(\gamma_{left}(a)-\theta))\\
=
&f(-\sin(\gamma_{left}(a)-\theta))+f(\sin(\gamma_{left}(a)-\theta))<0
\end{aligned}
\]
\[
I_{\rho,\theta}=2c<0
\]

(2) On $[c,c+2\pi)$, there is 4 zero points of $\beta(\alpha)-q_{\rho,\theta}(\alpha)$, and there are exactly 2 sign-changing ones.

\begin{figure}[htbp]
  \centering
  \includegraphics[width=0.6\textwidth]{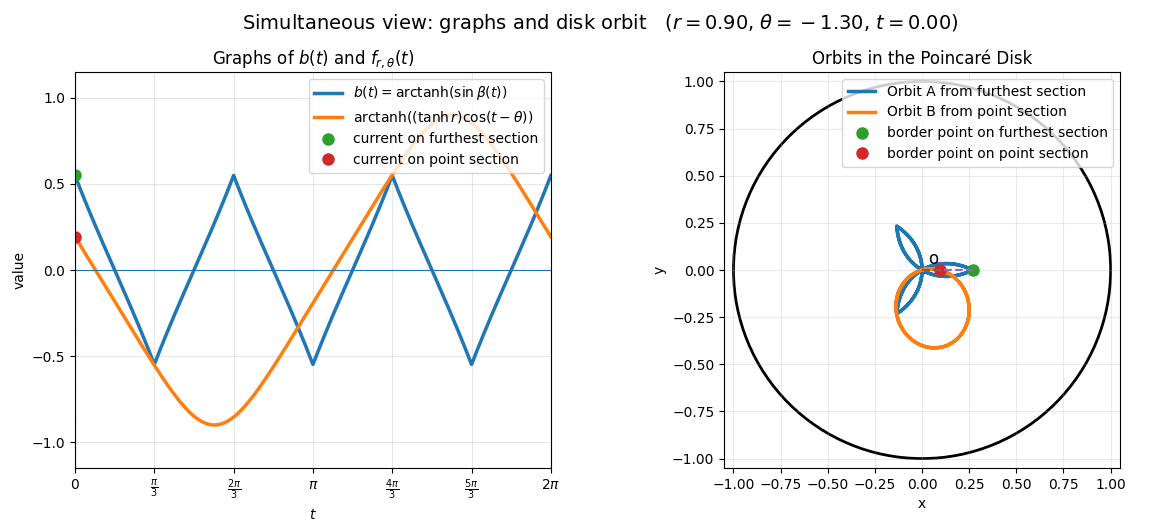}
  \caption{almost 4 zero points}
  \label{4 zero points}
\end{figure}

Let's assume $\beta_*$ increases at $c$ which is a sign-changing one. We have
$
\gamma_{left}(c)+\pi=\gamma_{right}(c+\pi)
$
There is exactly one sign-keeping zero point $c_*\in[c,c+\pi]$. In fact, $c_*$ is the $x$ coordinate of a downward peak of $\beta_*$, so \[c_*=\gamma_{left}(c_*)-\frac\pi6=\gamma_{right}(c_*)+\frac\pi6\]
\[
\begin{aligned}
c_0+c_1&=f(\sin(\gamma_{right}(c_*)-\theta))-f(\sin(\gamma_{left}(c)-\theta))\\&\,\,\,\,\,\,\,\,+f(\sin(\gamma_{right}(c+\pi)-\theta))-f(\sin(\gamma_{left}(c_*)-\theta))\\
&=-f(\sin(\gamma_{left}(c)-\theta))+f(\sin(\gamma_{left}(c)+\pi-\theta))\\
&=-f(\sin(\gamma_{left}(c)-\theta))+f(-\sin(\gamma_{left}(c)-\theta))<0
\end{aligned}
\]
\[I_{\rho,\theta}=2(c_0+c_1)<0\]

(3) On $[c,c+2\pi)$, there is 6 zero points of $\beta(\alpha)-q_{\rho,\theta}(\alpha)$

\begin{figure}[htbp]
  \centering
  \includegraphics[width=0.6\textwidth]{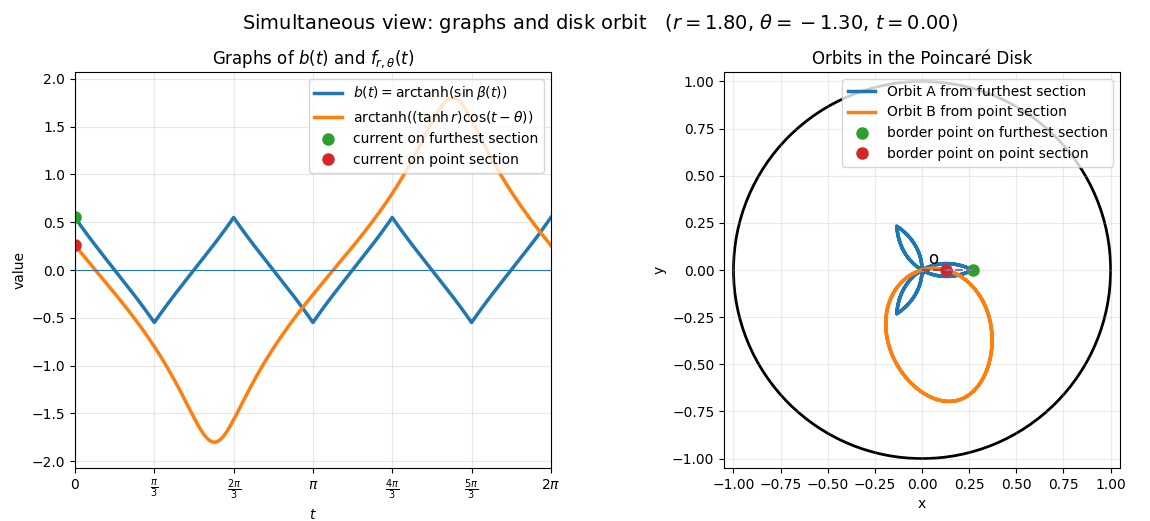}
  \caption{6 zero points}
  \label{6 zero points}
\end{figure}

Recall the monotonity, then ther are all sign-changing zero points, $a_0< b_0< a_1< b_1< a_2< b_2$. By anti-period $\pi$, we have
\[
b_1-a_0=a_2-b_0=b_2-a_1=\pi
\]

So the six terms $f(\sin(\frac\pi6\pm(m_j-\theta))), j\in{0,1,2}$
pair into some \[
(f(x_i),\,f(-x_i)),\,i\in{0,1,2}
\]

Here $x_i=\sin(\frac\pi6\pm(m_j-\theta))\in(-1,1)$, so 
\[\Sigma_jc_j=\Sigma_i(f(x_i)+f(-x_i))<0\]
\[I_{\rho,\theta}=2\Sigma_jc_j<0\]

\smallskip
Now that we prove $I_{\rho,\theta}<0$, so
\[d_\mathcal{M}(\sigma_p,\sigma_*)>d_\mathcal{M}(\sigma_o,\sigma_*)=12\ln2-6\ln3\]
\end{proof}

\section{Discussion}

\subsection{Medianization for Sphere}
There is a measured wall structure for $S^n$ embedded in $\mathbb R^{n+1}$ as $\{(x_1,\dots,x_n):x_1^2+\dots+x_n^2=1 \}$

For any point $p\in S^n$, we define the open half-space as $h_p=\{x\in S^n:\langle p,x\rangle>0\}$, then $ {h_p}^c\subset S^n$ is a closed half-space. The set of walls is $\mathcal W={\{h_p,{h_p}^c\},\ p\in S^n}$. The reason to define like this is that when $n\ge1$, a totally geodesic embedded $S^{n-1}$ in $S^n$ is in one-to-one correspondence to a $n$-dimensional hyperplane crossing the origin point and the $S^n$.

The G-invariant is immediate with the usual measure of $S^n$ for walls(half-spaces). The rest of the measured wall structure is obvious. And we have
\[
d_{S^n}(x,y) = \frac{\pi}{\text{vol}(S^n)}\mu_{S^n}H(x|y)
\]

The intersection condition for walls and inclusion condition for half-spaces are trivial: Every two distinct walls always intersect with each other when $n\ge2$, and never intersect when $n=1$; Every two distinct open half-spaces have no inclusion relation. So, every section that is symterical to its complement is an admissible section. Notice that the measure of wall-space is finite. As a result, the medianization is just the set of every section with measurable chosen open half-spaces symterical to the complement of chosen set.

A point section $\sigma_p$ is just represented by a open semi-shpere cap centered at the point $p$. The Hausdorff distance is
\begin{fact}
\[
d_{\mathcal M}(S^n,\mathcal M(S^n))
=
\frac{\pi}{2}
\]
\end{fact}

Notice that the point set representing chosen open half-spaces of an admissible section is always of measure $\frac1{2}{{\text{vol}(S^n)}}$, $d_{\mathcal M}(S^n,\mathcal M(S^n))\le\frac\pi2$ is obvious. 

The reason of $
d_\mathcal M(S^n,\mathcal M(S^n))
\ge
\frac{\pi}{2}
$
is that the equation is a limit of a sequence. For $S^1$ case, we nomalize it into an interval $[0,1)$ and do as Cantor set. But we don't delete parts here, instead we sent it to opposite in $S^1$, or equivalently $x\to (x+\frac12)  \text{ mod } 1$ in interval.

\begin{figure}[htbp]
  \centering
  \includegraphics[width=0.6\textwidth]{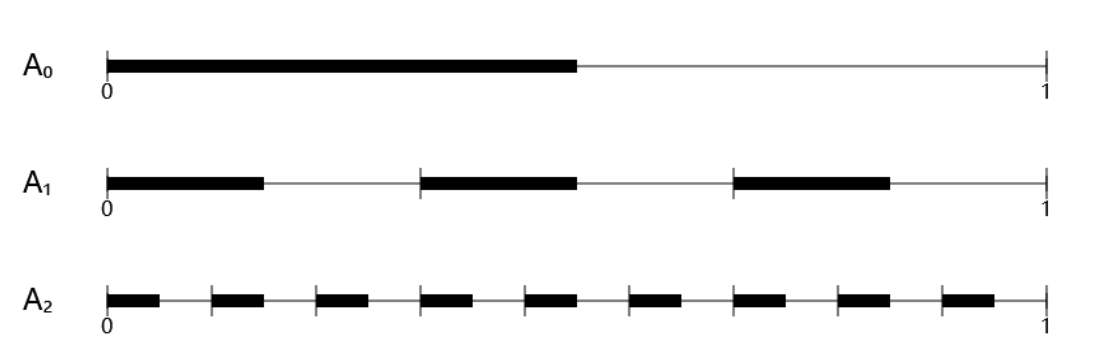}
  \caption{Sequence to the Equation}
  \label{Sequence to the Equation}
\end{figure}

Note that there is no limit set like the Cantor set. This case, and cases for $n>1$ are obvious then.

\subsection{Medianization for Euclidean Space}
There is a measured wall structure for $\mathbb{R}^n$ embedded in $\mathbb{R}^{n+1}$ as $\mathbb R^n_*=\{(x_1,\dots,x_{n+1}):x_1=1 \}$.

For any point $x\in \mathbb R^n$ with $x_2^2+\cdots+x_{n+1}^2\ne0$, we define the open half-space as $h_x=\{y\in \mathbb{R}_*^n:\ y\cdot x=x_1+x_2y_2+\dots+x_ny_n>0\}$, then $ {h_x}^c\subset \mathbb R^n_*$ is a closed half-space. The set of walls is $\mathcal W={\{h_x,{h_x}^c\},\ p\in \mathbb R^n_*}$. The reason to define like this is that when $n\ge1$, a totally geodesic embedded $R^{n-1}$ in $\mathbb R^n$ is in one-to-one correspondence to a $n$-dimensional hyperplane crossing the origin point and the $\mathbb R^n_*$.

The orthogonal complement of a such hyperplane intersects $S^n$ at exactly two opposite points, except for the north pole and south pole. So the space of open half-spaces of $\mathbb{R}^n$ is exactly $S^n-\{\text{north pole,south pole}\}$. It needs some work to see the measured wall structure. $\mathbb R^n$ has unbounded metric, but $S^{n}$ has finite measure. So we should move on to $\mathbb{R}\times S^{n-1} $ which is homeomorphic to $S^n-\{\text{north pole,south pole}\}$. We write $x=(x_1,\dots,x_n)\in \mathbb{R}\times S^{n-1}$, such that $x_2^2+\dots +x_{n+1}^2=1$, then this point $x$ is in one-to-one correspondence with open half-spaces of $\mathbb{R}^n_*$:
\[h_x=\{y\in \mathbb{R}^n_*:y\cdot x=x_1+x_2y_2+\dots+x_ny_n>0\}\]

To define a measure on the space of half-spaces, we just use the product measure of measures of $\mathbb{R}$ and $S^n$. That is, we denote $
x=(r,\omega)\in \mathbb R\times S^{n-1}
$, and the open half-space is $
h_{r,\omega}
=
\{y\in\mathbb R^n_*:\langle y,\omega\rangle+r>0\}.
$

We claim that the measure
$
d\mu(r,\omega)=dr\,d\omega,
$
where $dr$ is Lebesgue measure on $\mathbb R$ and $d\omega$ is the standard rotation-invariant measure on $S^{n-1}$, is invariant under the Euclidean group
$
G=\mathbb R^n\rtimes O(n).
$
That is, $
\mu(gE)=\mu(E),
$ for every measurable set $E\subseteq \mathbb R\times S^{n-1}$ and every $g\in G$. Equivalently,
\[
\int_{\mathbb R\times S^{n-1}} f(g\cdot \xi)\,d\mu(\xi)
=
\int_{\mathbb R\times S^{n-1}} f(\xi)\,d\mu(\xi)
\]
for every compactly supported measurable function $f$.

\begin{proof}

Let
$
g(y)=Ay+a,
$
where $A\in O(n)$ and $a\in\mathbb R^n$. A point $z$ belongs to $gh_{r,\omega}$ if and only if
$
z=Ay+a
$
for some $y\in h_{r,\omega}$.
That is,
\[
0
<
\langle A^{-1}(z-a),\omega\rangle+r=
\langle z-a,A\omega\rangle\\=
\langle z,A\omega\rangle
-\langle a,A\omega\rangle
+r.
\]
So
$
gh_{r,\omega}
=
h_{\,r-\langle a,A\omega\rangle,\;A\omega},
$
and the induced action on half-spaces is
\[
g\cdot(r,\omega)
=
\bigl(r-\langle a,A\omega\rangle,\;A\omega\bigr).
\]

Let $f$ be a compactly supported measurable function on $\mathbb R\times S^{n-1}$, then
\[
\int_{\mathbb R\times S^{n-1}}
f(g\cdot(r,\omega))
\,dr\,d\omega
=
\int_{\mathbb R\times S^{n-1}}
f\bigl(r-\langle a,A\omega\rangle,\;A\omega\bigr)
\,dr\,d\omega
\]
\[=
\int_{\mathbb R\times S^{n-1}}
f\bigl(r-\langle a,A\omega\rangle,\;A\omega\bigr)
\,d(r-\langle a,A\omega\rangle)\,d(A\omega)
=
\int_{\mathbb R\times S^{n-1}}
f(r,\omega)
\,dr\,d\omega.
\]
\end{proof}
We know that in Euclidean space, if two distinct codimension-1 hyperplanes are unparallel, then they intersects with each orther. $h_{p}\subset h_{q}$ if and only if $ \omega_p=\omega_q \text{ and } r_p<r_q$. So to determine an admissible section $\sigma$ is for every direction $\omega \in S^{n-1}$ to choose a dividing point $m(\omega)\in \mathbb{R}\cup \{\pm \infty\}$ for choosing a half-space. That is, view $\sigma$ as the half-spaces it choose, then $\sigma\cap(\mathbb{R}\times\{\omega\})=\{x>m(\omega)\}\times\{\omega\} $. For a point $y\in \mathbb R^n_*$, the point section is $\sigma_y=\left\{(r,\omega):\right \langle y,\omega\rangle+r>0\}$.

For $n=1$, there is only a direction, and an admissible section is to choose a point (maybe $\pm \infty$) to divide the choice, and the median section of three point section is the point section of the median point. We have
\[
d_\mathbb{R}(x,y) = \frac{1}{2}\mu_{\mathbb R\times \{p_1,p_2\}}H(x|y),
\quad
d_\mathcal M(\mathbb {R}^n,\mathcal M(\mathbb {R}^n))
=d_\mathcal M(\mathbb {I}^n,\mathcal M(\mathbb {I}^n))
=
0
\]
When $n\ge 2$, \[
d_{\mathbb{R}^n}(x,y) = \frac{4\text{vol}(S^{n-2})}{n-1}\mu_{\mathbb R\times S^{n-1}}H(x|y)
\]

The Hausdorff distance is infinite. For $n=2$ case, recall the construction ${A_n}$ for shperes $S^{n-1}$, and define $B_{m,n}$ in this way: $m(\omega)=m$ if $\omega\in A_n$, otherwise $m(\omega)=-m$. The rest is obvious.
\begin{fact}
\[
d_\mathcal M(\mathbb {R}^n,\mathcal M(\mathbb {R}^n))
=
+\infty\qquad(n\ge2)
\]
\end{fact}

\subsection{Easy Proof of Lemma 1.3}
We introduce another coordinate and offer a simpler proof of Lemma 1.3. 
We use that
$
\mathrm{dS}^n \cong \mathbb R\times S^{n-1},
$ by
\[
(\sinh r,\cosh r\,\omega)\leftrightarrow (\sinh r,\omega),\, r\in \mathbb{R},\, \omega\in S^{n-1}
\]

The $SO(n,1)$-invariant measure on the space of $\mathrm{dS}^n$ is
\[
d\mu(r,\omega)
=
\cosh^{\,n-1}(r)\,dr\,d\omega,
\]
where $d\omega$ denotes the standard volume measure on $S^{n-1}$.

Let
$
p=(x_0,x_1,\ldots,x_n)\in \mathbb{I}^n,
\ 
x_0^2-\sum_{i=1}^{n}x_i^2=1.
$
The set of all hyperplanes passing through $p$ is bounded by the following set of open half-spaces
\[
\left\{
(r,\omega)\in \mathbb R\times S^{n-1}
:
r
=
\operatorname{arctanh}
\!\left(
\frac{\sum_{i=1}^{n}x_i\omega_i}{x_0}
\right)
\right\}
\]

Then we proof Lemma 1.3 in a simple way.
\begin{proof}
Let
$
p_0=(1,0,\dots,0)
$
minimizing
$
D(p)=d_{\mathcal M}(\tau,\sigma_p).
$

Move \(p_0\) a small distance \(s\) in direction
$
\xi\in S^{n-1},
$
then
$
p_s =(\cosh s,\sinh s\,\xi).
$

For a wall
$
u(r,\omega)=(\sinh r,\cosh r\,\omega),
$
the new point section \(\sigma_{p_s}\) has dividing point \(a_s(\omega)\) determined by
$
\langle p_s,u(r,\omega)\rangle=0.
$
Thus
\[
-\cosh s\sinh r+\sinh s\cosh r\,\langle \xi,\omega\rangle=0
\implies
a_s(\omega)
=
\operatorname{artanh}
\big(\tanh s\,\langle \xi,\omega\rangle\big)
\]

For small \(s\),
$
a_s(\omega)=s\langle \xi,\omega\rangle+O(s^3),
$ so
\[
D(p_s)
=
\int_{S^{n-1}}
\left|
\int_{a_s(\omega)}^{b(\omega)}
\cosh^{n-1}r\,dr
\right|
d\omega
\]

Assuming \(b(\omega)\neq 0\) almost everywhere, we get at \(s=0\) 
\[
\frac{d}{ds}\bigg|_{s=0}D(p_s)
=
-\int_{S^{n-1}}
\operatorname{sgn}(b(\omega))
\langle \xi,\omega\rangle
\,d\omega
\]

So \(p_0\) minimizes \(D\) only if this vanishes for every
\(\xi\in S^{n-1}\). Equivalently,
\[
\int_{S^{n-1}}
\operatorname{sgn}(b(\omega))\,\omega\,d\omega
=
0
\]
\end{proof}

\subsection{The Nowhere-Zero Assumption of Border Function}
To make the locally minimum condition easy, we assumed that the border function is not zero almost everywhere on $S^{n-1}$. Now we prove that this assumption doesn't change our calculation of the Hausdorff distance.

The idea is that when $n\ge2$ we can interrupt the border function a bit while keeping the border function satisfying the locally minimum condition. The interruption is small, so the $\beta$ is still 1-Lipschitz, and the interrupted section is arbitrarily close to the original one. By triangle inequation, its distance to a fixed point section changes arbitrarily small and controled by the interruption uniformly, so as its distance to the embedded $\mathbb{I}^n$. The Hausdorff distance is a supreme over the set of some sections, so by calcyulating over sets of almost nowhere-zero border function we get the same Hausdorff distance.

Firstly we prove that the interruption is possible. The first step is that an interruption is possible on an open sphere where it is zero. Let's assume the open sphere $S$ centers at $(1,0,\dots ,0)$, then by oddness the bordered function is zero on $-S$ who centers at $(-1,0,\dots ,0)$. The $n$ axises or coordinates divide $S\cup -S$ into $2^n$ symmetrical components. We note one of them as $C$. We can interrupt the border function on $C$ a bit, so that it keep $\beta$ 1-Lipschitz, and is not zero on C, and is zero on border of $C$. By symmetry we copy the interruption to other components up to sign. Then we decomposite the locally minimum condition into coordinates, and will find each interruption contribute to a fix coordinate with the same value up to sign. The next step is to change the sign of these interruption, so that the border function remains odd, and the $\beta$ remain 1-Lipchitz, and the locally minimum condition holds. This is possible due to the following fact.
\begin{fact}
Let \(n \ge 3\). There exists a map
$
f:\{-1,1\}^n \to \{-1,1\}
$
such that
\[
f(a)=-f(-a),
\qquad
\sum_{a\in\{-1,1\}^n} f(a)a=(0,\dots,0)
\]
\end{fact}

\begin{proof}
$
f(a_1,\dots,a_n)=a_1a_2a_3$ satisfies these conditions.
Oddness is obvious.

For the \(j\)-th coordinate, we have
\[
\sum_{a\in\{-1,1\}^n} f(a)a_j
=
\sum_{a\in\{-1,1\}^n} a_1a_2a_3a_j.
\]

If \(j=1\), it becomes
$
\sum_{a\in\{-1,1\}^n} a_2a_3=0.
$

If \(j=2\), it becomes
$
\sum_{a\in\{-1,1\}^n} a_1a_3=0.
$

If \(j=3\), it becomes
$
\sum_{a\in\{-1,1\}^n} a_1a_2=0.
$

If \(j>3\), it becomes
$
\sum_{a\in\{-1,1\}^n} a_1a_2a_3a_j=0.
$
\end{proof}
The value of $f$ indicate the sign of the corresponding interruption. For $n=2$, such $f$ doesn't exist, and $S$ is an arc. We divide $S$ and $-S$ into 8 arcs averagely, and assing sign \textbf{+ - + -} on these arcs, and adjust the ratio. Notice that we can change the interruptions by a same small positive, so that $\beta$ is still 1-Lipschitz, and the distance between it and the original section is very small.

Now that we can do a such arbitrarily small interruption on a sphere, we can repeat it on the inner of thw set where the border function is zero, and take the limit border function. The corresponding section functino is what we want.
For $n=1$, the case is easy and we don't need this interruption.
\subsection{The Gram Matrix}

Let \(V\) be a finite-dimensional real vector space equipped with a symmetric bilinear form
\(\langle\cdot,\cdot\rangle\) (not necessarily positive definite).
For an ordered set of vectors \(v_1,\dots,v_k \in V\), their \emph{Gram matrix} is the
\(k\times k\) matrix
$
G(v_1,\dots,v_k) = \bigl( \langle v_i, v_j \rangle \bigr)_{i,j=1}^k .
$

In our setting, \(V = \mathbb{R}^{n+1}\) with
$
\langle x,y\rangle = -x_0y_0 + x_1y_1 + \cdots + x_ny_n .
$

For two vectors \(u,v \in \mathrm{dS}^n \subset \mathbb{R}^{n+1}\) we get
\[
G(u,v) =
\begin{pmatrix}
\langle u,u\rangle & \langle u,v\rangle \\[2pt]
\langle v,u\rangle & \langle v,v\rangle
\end{pmatrix}
=
\begin{pmatrix}
1 & a \\
a & 1
\end{pmatrix},
\qquad a = \langle u,v\rangle .
\]

\begin{prop}[Basic Properties of the Gram Matrix]
Let \(G\) be the Gram matrix of \(v_1,\dots,v_k\).
\begin{enumerate}
\item \emph{Symmetry.} \(G\) is symmetric because the bilinear form is symmetric:
\(\langle v_i,v_j\rangle = \langle v_j,v_i\rangle\).
\item \emph{Change of basis.} If we express new vectors as combinations
\(w_i = \sum_{j=1}^k c_{ij} v_j\), then
\[
\langle w_i,w_j\rangle = \sum_{p,q=1}^k c_{ip} c_{jq} \langle v_p,v_q\rangle .
\]
In matrix form, \(G_{\text{new}} = C\, G\, C^{\mathsf{T}}\), where \(C = (c_{ij})\).
\item \emph{Degeneracy.} \(\det G = 0\) if and only if the subspace \(\operatorname{span}\{v_1,\dots,v_k\}\) is degenerate. That is, there is a vector $v_0$ in it such that $\langle v_0,v_1\rangle=\dots=\langle v_0,v_k\rangle=0$.
\end{enumerate}
\end{prop}

The signature of a Gram matrix completely tells the signature of the restriction of the bilinear form on the subspace spanned by the vectors. This is by Sylvester's law of inertia.

\begin{thm}
Let \(\langle\cdot,\cdot\rangle\) be a \textbf{non-degenerate} symmetric bilinear form on
\(\mathbb{R}^N\) of signature \((p,q)\) with \(p+q = N\). Let \(S = \operatorname{span}\{v_1,\dots,v_k\}\) and let \(G\) be the Gram matrix of \(\{v_1,\dots,v_k\}\). Then the restriction \(\langle\cdot,\cdot\rangle|_S\) has signature \((r,s,t)\) (with \(r+s+t=k\), where \(t\) is the degeneracy) exactly when \(G\) has \(r\) positive, \(s\) negative and \(t\) zero eigenvalues (counted with multiplicity).
\end{thm}

\begin{proof}
Consider the restriction of the bilinear form to \(S\). For any vectors
\(x = \sum_{i=1}^k x_i v_i\) and \(y = \sum_{i=1}^k y_i v_i\) in \(S\) we have
\[
\langle x,y\rangle = \sum_{i,j=1}^k x_i y_j \langle v_i,v_j\rangle
= \mathbf{x}^{\mathsf{T}} G \mathbf{y},
\]
where \(\mathbf{x} = (x_1,\dots,x_k)^{\mathsf{T}}\) and similarly for \(\mathbf{y}\).

Now choose any other basis \(\mathcal{B}'\) of \(S\). If \(P\) is the change-of-basis
matrix from \(\mathcal{B}\) to \(\mathcal{B}'\), then the matrix of the form in the new
basis is
$
G' = P^{\mathsf{T}} G P.
$

By Sylvester's law of inertia, $G$ and $G'$ have the same numbers of positive, negative and zero eigenvalues. The signature of the bilinear form on \(S\) is the triple \((r,s,t)\) of positive, negative and zero entries in any diagonal matrix representing the form in an orthogonal basis. By diagonalisation of a symmetric matrix, there exists an invertible matrix \(P\) such that\(P^{\mathsf{T}} G P\) is diagonal with entries \(1,-1,0\). The numbers of such entries are precisely the numbers of positive, negative and zero eigenvalues of \(G\). By Sylvester's law guarantees the \((r,s,t)\) is unique.
\end{proof}

In particular, if \(S\) is non-degenerate, its signature equals the numbers of
positive and negative eigenvalues of \(G\). We also need the behaviour of the signature when passing to the orthogonal complement.

\begin{thm}[Signature of the Orthogonal Complement]
Equip \(\mathbb{R}^N\) with a non-degenerate symmetric bilinear form of
signature \((p,q)\). Let \(S\) be a non-degenerate subspace of dimension \(k\) and
signature \((r,s)\) with \(r+s = k\) (so \(r\) positive, \(s\) negative eigenvalues for
the restriction). Then the orthogonal complement
\[
S^\perp = \{ x\in\mathbb{R}^N : \langle x,v\rangle = 0 \text{ for all } v\in S\}
\]
is also non-degenerate and has signature
$
(p-r,\; q-s).
$
\end{thm}

\begin{proof}
By the lemma below, the form is non-degenerate on $S$ and $S^\perp$, and $
\mathbb{R}^N = S \oplus S^\perp.$

By non-degenerating, we can take a basis of \(S\) so that the Gram matrix of the basis on $S$ is
\(\operatorname{diag}(\underbrace{1,\dots,1}_{r},\underbrace{-1,\dots,-1}_{s})\). Same for $S^\perp$. Since the two subspaces are
orthogonal, the union of two bases is a basis
of \(\mathbb{R}^N\) with block-diagonal Gram matrix
\[
\begin{pmatrix}
I_{r,s} & 0 \\
0 & I_{r',s'}
\end{pmatrix},
\]
And we just count along the diagocal and get
\[
r+r'=p,\qquad
s+s'=q
\]
\end{proof}

\begin{lem}
If a symmetric bilinear form on a finite-dimensional vector space \(V\) is non-degenerate and \(U\) is a non-degenerate subspace, then \(U^\perp\) is also non-degenerate and \(V = U \oplus U^\perp\).
\end{lem}

\begin{proof}
For any \(v\in V\), there exists a unique \(u\in U\) with \(\langle u,x\rangle = \langle v,x\rangle\) for all \(x\in U\), otherwise the form degenerates on $U$. Then
\(v-u \in U^\perp\) and \(v = u + (v-u)\). Thus \(V = U + U^\perp\). The sum is direct
because \(U \cap U^\perp = \{0\}\) by non-degeneracy of \(U\).

Now suppose \(w \in U^\perp\cap {U^\perp}^\perp\). Then for any
\(v = u + w' \in V\) with \(u\in U,\,w'\in {U^\perp}\) we have \(\langle w,v\rangle = \langle w,u\rangle + \langle w,w'\rangle
= 0+0 = 0\). Since the whole form is
non-degenerate, \(w = 0\). So \(U^\perp\) is non-degenerate.
\end{proof}

For the wall intersection condition, we apply these facts as follows:
\(\mathbb{R}^{n+1}\) has signature \((n,1)\). For \(u,v\in\mathrm{dS}^n\), the subspace \(S = \operatorname{span}\{u,v\}\) has Gram matrix \(\begin{pmatrix}1&a\\a&1\end{pmatrix}\) with \(a = \langle u,v\rangle\).

If \(|a|<1\), then \(\det G = 1-a^2 >0\) and the diagonal entries are positive,
so \(G\) is positive definite. So \(S\) has signature \((2,0)\). By the orthogonal complement theorem, \(S^\perp\) has signature \((n-2,1)\) and therefore contains a hyperbolic point. The hyperplanes intersect(including the identity case).

If \(|a|=1\), then \(\langle u,v\rangle=\pm1\). When \(u=\pm v\) they share the same hyperplanes. When there is a $x\in\mathbb{I}^n$ such that $\langle u,x\rangle=\langle v,x\rangle=0$, then $\langle u\mp v,x\rangle=0$, so $(u-v)\in (\operatorname{span}(x)^\perp)$. Plus we have $\langle u\mp v,u\mp v\rangle=0$. Notice that the form has signature (0,1) on $\operatorname{span}(x)$, so by the theorem we just proved the form is non-degenerate on $(\operatorname{span}(x)^\perp)$. So $u\mp v=0$. So the hyperplanes intersect if and only if they are the same in this case.

If \(|a|>1\), then \(\det G <0\), so \(G\) has one negative eigenvalue (signature \((1,1)\)). Then \(S^\perp\) has signature \((n-1,0)\) and thus cannot have a hyperbolic point. So the hyperplanes do not intersect.

\subsection{Summary of Measures and Metrics}
We have some spaces with a (pseudo-)metric or (pseudo-)measure.
\begin{enumerate}
\item $n$-dimensional real hyperbolic space $\mathbb{I}^n$ has a hyperbolic distance $d_{\mathbb{I}^n}$.
\item The measured walls structure of it has a measure $\mu_\mathcal{W}$.
\item The space of open half-spaces $\mathcal{H}$ has a pulled back measure $\mu_\mathcal{H}$.
\item The de Sitter space of dimension $n$ has a measure $\mu_{\mathrm{dS}^n}$.
\item The space of sections has a metric $d_{Section}$. The medianization $\mathcal{M}(\mathbb{I}^n)$ has a metric $d_\mathcal{M}$.
\end{enumerate}
They are strongly related.
\begin{enumerate}
\item $d_{\mathbb{I}^n}$ is the usual hyperbolic distance. $\mu_{\mathrm{dS}^n}$ is defined.
\item We prove that the $\mathrm{dS}^n$ is a model for $\mathcal{W}$. \[\frac{2\,\operatorname{vol}(S^{n-2})}{n-1}d_\mathcal{W}=\mu_{\mathrm{dS}^n}\]
\item For points $x,y\in \mathbb{I}^n$, we have \[d_{\mathbb{I}^n}(x,y)=\mu_\mathcal{W}(W(x|y))=\frac12\mu_{\mathcal H}(H(x|y))\]
\item For sections $\sigma_1,\sigma_2$, we also separately use these notations to note the open half-spaces they choose, then \[d_{Section}(\sigma_1,\sigma_2)=2\ \mu_\mathcal{H}(\sigma_1\triangle\sigma_2)\]
\item $\mathbb{I}^n$ is embedded isometrically into $\mathcal{M}(\mathbb{I}^n)$. For point $x,y\in \mathbb{I}^n$, we have
\[
d_{\mathbb{I}^n}(x,y)=d_\mathcal{M}(x,y))=d_{\text{Section}}(\sigma_x,\sigma_y)
\]
\end{enumerate}

\end{document}